\documentclass[12pt]{amsart}    
\usepackage{latexsym}
\usepackage{amssymb}
\usepackage{amscd}
\usepackage{longtable}

\newcommand{\Right}{\mathcal{R}}
\newcommand{\Left}{\mathcal{L}}
\newcommand{\E}{\mathfrak{exp}}
\newcommand{\C}{\mathfrak{con}}
\newcommand{\sat}{\operatorname{sat}}
\newcommand{\M}{\mathcal M}
\newcommand{\Set}{\mathcal S}
\newcommand{\mif}{\mbox{if} ~}

\DeclareMathOperator{\gin}{gin}
\DeclareMathOperator{\reg}{reg}
\DeclareMathOperator{\charf}{char}

\numberwithin{equation}{section}

\theoremstyle{plain}   
\begingroup 

\newtheorem{thm}{\bf Theorem}[section]
\newtheorem{cor}[thm]{Corollary}     
\newtheorem{lem}[thm]{Lemma}         
\newtheorem{prop}[thm]{Proposition}  

\theoremstyle{definition}
\newtheorem{defn}[thm]{Definition}   
\newtheorem{rem}[thm]{Remark}        
\newtheorem{remdefn}[thm]{Remark and Definition}   
\newtheorem{ex}[thm]{Example}        
\newtheorem{alg}[thm]{Algorithm}     
\newtheorem{question}[thm]{Question} 

\endgroup

\title{Algorithms for strongly stable ideals}

\author{Dennis Moore}
\author{Uwe Nagel}
\address{Department of Mathematics,
University of Kentucky, 715 Patterson Office Tower,
Lexington, KY 40506-0027, USA}
\email{d.k.moore@uky.edu}
\email{uwe.nagel@uky.edu}

\thanks{This work was partially supported by a grant from the Simons 
Foundation (\#208869 to Uwe Nagel). The authors were also partially 
supported by the National Security Agency under Grant Number H98230-09-1-0032.}

\begin{document}



\begin{abstract}
Strongly stable monomial ideals are important in Algebraic Geometry,
Commutative Algebra, and Combinatorics. Prompted, for example, by
combinatorial approaches for studying Hilbert schemes and the
existence of maximal total Betti numbers among saturated ideals with
a given Hilbert polynomial, in this note we present three algorithms
to produce {\em all} strongly stable ideals with certain prescribed
properties: the saturated strongly stable ideals with a given
Hilbert polynomial, the almost lexsegment ideals with a given
Hilbert polynomial,  and the saturated strongly stable ideals with a
given Hilbert function. We also establish results for estimating the
complexity of our algorithms.
\end{abstract}

\maketitle


\section{Introduction}

Strongly stable monomial ideals arise naturally in Algebraic
Geometry,  Commutative Algebra, and Combinatorics. In fact, Galligo,
Bayer and Stillman showed that the generic initial ideal of a
homogeneous ideal is Borel-fixed. In characteristic zero,
Borel-fixed ideals are strongly stable (see, e.g., \cite{DE} or
\cite{HeHi}).  Shifting is a combinatorial technique that studies a
given simplicial complex by modifying the given complex to a simpler
one while preserving essential properties. Strongly stable ideals
figure prominently in the algebraic approach to shifting (see, e.g.,
\cite{HeHi}). A Hilbert scheme parametrizes the closed subschemes of
a projective space with a fixed Hilbert polynomial. Its scheme
structure is very complex. Strongly stable ideals are the basis for
combinatorial approaches for studying Hilbert schemes (see, e.g.,
\cite{RH}, \cite{AR}, \cite{RS}, \cite{NS}).

Building on work by Reeves \cite{AR}  and Gehrs \cite{KG}, in this
note we present  an algorithm that produces {\em all} saturated
strongly stable ideals of a polynomial ring with a given Hilbert
polynomial. We restrict ourselves to saturated ideals for two
reasons. With respect to the reverse lexicographic order, the
generic initial ideal of an ideal is saturated if and only if the
ideal is saturated, and the homogeneous ideal of a closed subscheme
is saturated. Moreover, the number of strongly stable ideals with a
given Hilbert polynomial is not finite.

We also develop two related algorithms. Recently, Caviglia and Murai
(see \cite{CM}) established that in the set of all saturated
homogeneous ideals of a polynomial ring  with a given Hilbert
polynomial there exists an ideal whose total Betti numbers are at
least as large as the total Betti numbers of all other ideals in
this set. This generalizes a result of Valla (see \cite{GV}) about
ideals with constant Hilbert polynomial. Thanks to a result by
Bigatti, Hulett, and Pardue, there must be an ideal with maximal
Betti numbers that is saturated, strongly stable, and a lexsegment
ideal when considered in polynomial ring in one less variable. We
call such an ideal almost lexsegment (see Definition
\ref{almost-lex-ideal}). We show that a modification of our first
algorithm produces all almost lexsegment ideals to a given Hilbert
polynomial by computing only almost lexsegment ideals at every step.
The algorithm reveals in particular that, for a given Hilbert
polynomial, there can be many almost lexsegment ideals that achieve
the maximal Betti numbers.

Our third algorithm produces all saturated strongly stable ideals
with a given Hilbert function. They form a subset of the ideals
obtained by the first algorithm. However, we present a more direct
and more efficient algorithm for computing them.

This note is organized as follows. In Section \ref{sec:basics},  we
briefly recall some well-known properties of strongly stable ideals.
For unexplained terminology and background, we refer to \cite{DE},
\cite {HeHi}, and \cite{RH2}.

In Section \ref{sec:exp} we introduce certain algorithmic
operations ---  called contractions and expansions of monomials ---
on the set of minimal generators of strongly stable ideals. These
operations were first proposed in \cite{AR} and also considered in
\cite{KG}. For greater efficiency, we use suitable modifications of
these operations,  and we describe their effect on the Hilbert
polynomial.

The theoretical core for our algorithms is provided by Theorem
\ref{link-all-ideals}. It states that all saturated strongly stable
ideals with the same Hilbert polynomial can be computed by using
expansions of minimal monomial generators. The proof of this result
is constructive and leads to a new algorithm for finding all
saturated ideals having a prescribed Hilbert polynomial (see
Algorithm \ref{alg-stable-Hilbpoly}). It also includes a sharp
estimate on the number of steps the algorithm needs to generate a
strongly stable ideal starting from a trivial ideal.

Algorithm \ref{alg-stable-Hilbpoly} is  modified in Section
\ref{almost-lex-section} in order to produce all almost lexsegment
ideals to a given Hilbert polynomial (see Algorithm
\ref{alg-lex-Hilbpoly}). These ideals represent all the Hilbert
functions of saturated homogeneous ideals with the given Hilbert
polynomial. We also present an algorithm for directly generating all
saturated strongly stable ideals with a fixed Hilbert function (see
Algorithm \ref{alg-Hilbert-series}).

In Section \ref{sec:final}
we  discuss consequences of the complexity estimate in Theorem 
\ref{link-all-ideals}. In particular, we show that the number of 
saturated strongly stable ideals in a polynomial ring in $n$ 
variables with a given Hilbert polynomial $p$ does depend only on $p$ 
and not on $n$, once $n$ is sufficiently large (see Proposition 
\ref{min-num-vars}). Fixing the Hilbert polynomial, we also describe 
the ideals with the worst Castelnuovo-Mumford regularity (see Theorem 
\ref{thm:worst-reg}).

We implemented all algorithms presented in this note in the computer 
algebra system \textit{Macaulay2} \cite{GS}. The files can be 
downloaded at {\tt http://www.ms.uky.edu/$\sim$dmoore/M2}.

\section{Strongly stable ideals and some properties}
\label{sec:basics}

Throughout this note we denote by $R:= K[x_0,\ldots,x_n]$ the polynomial 
ring over an arbitrary field $K$. Also, we denote by $R^{(1)} := K[x_0, 
\ldots, x_{n-1}]$ the polynomial ring where the last variable has been 
removed, and, more generally, $R^{(j)} := K[x_0, \ldots, x_{n-j}]$ is 
the polynomial ring where the last $j$ variables have been removed. We 
use multi-index notation: If $A = (a_0,\ldots,a_n)$ is an $n$-tuple of 
non-negative integers we set $x^A = x_0^{a_0} \cdot \ldots \cdot 
x_n^{a_n}$. Moreover, if $x^A \neq 1$, the {\em max index} of $x^A$ is
$$ \max (x^A):= \max \{i : a_i > 0 \} = \max \{ i : x_i | x^A \}. $$

\begin{defn} \label{stable-ideals}
A monomial ideal $I \subset R$ is called a \emph{strongly stable ideal} 
if, for every monomial $x^A \in I$ and $x_j | x^A$, we have $\frac {x_i} 
{x_j} \cdot x^A \in I$ whenever $0 \leq i < j$.
\end{defn}

\begin{remdefn} \label{double-sat}
Let $I \subset R$ be a strongly stable ideal.
\begin{itemize}
\item[(i)] To determine whether an ideal is strongly stable, it is 
sufficient to check that the minimal monomial generators of  the ideal 
satisfy the criterion in Definition \ref{stable-ideals}.
\item[(ii)] The { saturation} of $I$ is the ideal $\sat_{x_n}(I)$ that 
is obtained from $I$ by setting $x_n=1$ in every monomial of $I$.
\item[(iii)] The \emph{double saturation} of $I$ is the extension ideal 
$\sat_{x_{n-1}, x_n}(I)$ in $R$ of the saturation of $\sat_{x_n}(I) 
\cap R^{(1)} \subset R^{(1)}$.  It is obtained from $I$ by setting 
$x_n = x_{n-1} = 1$.
\end{itemize}
\end{remdefn}

Throughout this note we use the lexicographic  order, $>_{lex}$, for 
comparing monomials of a given degree. Let $x^B = x_0^{b_0} x_1^{b_1} 
\cdots x_n^{b_n}$ and $x^C = x_0^{c_0} x_1^{c_1} \cdots x_n^{c_n}$ be 
two monomials of $R$ of the same degree. Recall that $x^B >_{lex} x^C$, 
if the first nonzero entry of the vector $(b_0 - c_0, \; b_1 - c_1, \; 
\ldots, \; b_n - c_n)$ is positive.

If $A$ is a graded $K$-algebra we denote its  Hilbert function by $h_A 
(j) = \dim_K [A]_j$, its Hilbert polynomial by $p_A$, and its Hilbert 
series by $H_A = \sum_{j \geq 0} h_A (j) \cdot t^j$. The Hilbert 
functions of graded $K$-algebras are completely classified. In 
particular, if $h$ is such a Hilbert function with $h (1) \leq n+1$, 
then there is a lexsegment ideal $L_h \subset R$ such that, for every 
integer $j$,\; $h_{R/L_h} (j) = h(j)$. Recall that a lexsegment ideal 
is a monomial ideal $I \subset R$ such that, for every integer $j$,\; 
$[I]_j$ is spanned by the first $\dim_K [I]_j$ monomials of $[R]_j$ in 
the lexicographic order. Lexsegment ideals are strongly stable.

At times we will abuse language and say that a  homogeneous ideal $I$ 
of $R$ has Hilbert function $h$ or Hilbert polynomial $p$ if $R/I$ has 
this Hilbert function or polynomial.

Let $p \in \mathbb{Q}[z]$ be the Hilbert polynomial of a  standard
graded $K$-algebra of dimension $d+1>0$. Then there are unique
integers $b_0 \geq b_1 \geq b_2 \geq \ldots \geq b_d > 0$ such that
\begin{equation}
\label{formula-Hilbert-polynomial}
p(z) = \sum_{i=0}^{d} \left[\binom{z+i}{i+1}-\binom{z+i-b_i}{i+1}\right].
\end{equation}

With respect to this representation, we recall the  lexicographic
ideal associated to a given Hilbert polynomial as introduced by
Macaulay. Some of the properties of this ideal have been studied by
Bayer in \cite{DB}. The lexicographic ideal is called a
\emph{universal lex} ideal in \cite{MH} and \cite{CM}. In order to
keep this note more self-contained and for the convenience of the
reader we provide short proofs for the results below.

\begin{thm} \label{L_p}
Let $p \neq 0$ be a Hilbert polynomial of a quotient of $R$. Then 
there is a unique saturated lexsegment ideal $L_p \subset R$ such that 
the Hilbert polynomial of $R/L_p$ is $p$. It is called the {\em 
lexicographic ideal} to $p$. The ideal $L_p$ is generated by the set 
of monomials
\begin{align*}
\{& x_0,x_1,\ldots,x_{n-d-2}, x_{n-d-1}^{a_d+1},
 x_{n-d-1}^{a_d} \cdot x_{n-d}^{a_{d-1}+1},\\
& x_{n-d-1}^{a_d} \cdot x_{n-d}^{a_{d-1}} \cdot x_{n-d+1}^{a_{d-2}+1},
\ldots,\\
& x_{n-d-1}^{a_d} \cdot x_{n-d}^{a_{d-1}} \cdot x_{n-d+1}^{a_{d-2}}
\cdot \ldots \cdot x_{n-3}^{a_2} \cdot x_{n-2}^{a_1+1}, \\
& x_{n-d-1}^{a_d} \cdot x_{n-d}^{a_{d-1}} \cdot x_{n-d+1}^{a_{d-2}}
\cdot \ldots \cdot x_{n-2}^{a_1} \cdot x_{n-1}^{a_0} \},
\end{align*}
where $p$ is written as in Equation $\eqref{formula-Hilbert-polynomial}$ 
and $a_d:= b_d, a_{d-1}:= b_{d-1} - b_d, \ldots, a_0:= b_0 - b_1$ 
(thus, $b_i = a_d + a_{d-1} + \ldots + a_i$), $0 \leq i \leq d$.
\end{thm}

\begin{proof} Because of its importance and for the convenience of the 
reader we include a proof. Set $L(a_0,\ldots,a_d) := L_p$. It is 
clearly a lexsegment ideal and saturated. We use induction on $d \geq 
0$ in order to compute the Hilbert polynomial of the quotient. If 
$d=0$, then we have $R/L(a_0) = K[x_0,\ldots,x_n]/(x_0,\ldots,x_{n-2},
x_{n-1}^{a_0}) \cong K[x_{n-1},x_n] / (x_{n-1}^{a_0})$, and  thus the 
Hilbert polynomial is $p_{R/L(a_0)}(z)=a_0=\binom{z}{1}-\binom{z-a_0}
{1}=p$, as claimed.

Let $d > 0$. Then multiplication by $x_{n-d-1}^{a_d}$ provides the 
exact sequence
$$ \begin{CD}
0 \rightarrow (R/L(a_0,\ldots,a_{d-1}))(-a_d) @>x_{n-d-1}^{a_d}>>
R/L(a_0,\ldots,a_d) \rightarrow R/(x_0,\ldots,x_{n-d-2},x_{n-d-1}^{a_d})
\rightarrow 0.
\end{CD} $$
Using the induction hypothesis we conclude that $p_{R/L(a_0, \ldots 
,a_{d})}=p$.

The uniqueness statement follows from the fact that $L_p$ is a 
lexsegment ideal and saturated.
\end{proof}

Note that the set of generators of the lexicographic ideal $L_p$ 
given in Theorem \ref{L_p} is not minimal when $a_0 = 0$.

The ideal $L_p$ has alternative characterizations.

\begin{prop} \label{prop-smallest-Hilb}
\begin{itemize}
\item[(a)] Let $L_h \subset R$ be a lexsegment ideal with Hilbert 
polynomial $p$, i.e., if $j \gg 0$, then $p(j) = h(j)$. Then the 
saturation of $L_h$ is the ideal $L_p \subset R$.
\item[(b)]  Let $R/I$ be a  graded quotient  of $R$ with Hilbert 
polynomial $p$. Then, for all integers $j$:
$$ h_{R/I} (j) \geq h_{L_p} (j). $$
\end{itemize}
\end{prop}

\begin{proof}
(a) Since $L_h$ and $L_p$ are both lexsegment ideals  and $h(j)=p(j)$ 
whenever $j \gg 0$, we get
$$ [L_h]_j = [L_p]_j, $$
whenever $j \gg 0$. As the ideal $L_p$ is saturated, it follows that 
$L_p$ is the saturation of $L_h$. (b) Denote by $h$ the Hilbert 
function of $R/I$. Then part (a) implies $L_h \subset L_p$, and the 
claim follows.
\end{proof}

We conclude this section with  formulae for certain invariants of 
stable ideals (in particular, strongly stable ideals), which will be 
useful later.  Note that these invariants only depend on the max 
indices and the degrees of the minimal generators of the ideal.

\begin{rem} \label{Hilbert-polynomial-and-series}
If $I \subset R$ is a saturated strongly stable ideal with minimal 
monomial generators $\{x^{A_1}, \ldots, x^{A_r}\}$, then let 
$l_i=\max(x^{A_i})$ and $d_i=\deg(x^{A_i})$, for all $1 \leq i \leq r$. 
The Hilbert polynomial and (nonreduced) Hilbert series of $R/I$ are
\begin{equation} \label{compute-Hilbert-polynomial}
p_{R/I}(z)=\binom{z+n}{n}-\sum_{i=1}^r \binom{z+n-d_i-l_i}{n-l_i}
\end{equation}
and
\begin{equation} \label{compute-Hilbert-series}
H_{R/I}(t)=\left(1-\sum_{i=1}^{r}(1-t)^{l_i}t^{d_i}\right)(1-t)^{-n-1}.
\end{equation}
The total Betti numbers of the ideal $I$ are
\begin{equation} \label{compute-Betti-numbers}
\beta_j(I)=\sum_{i=1}^r\binom{l_i}{j}.
\end{equation}
Equations \eqref{compute-Hilbert-series} and 
\eqref{compute-Betti-numbers} follow from the Eliahou-Kervaire 
resolution for stable monomial ideals (see \cite{EK}, p.\ 16); 
Equation \eqref{compute-Hilbert-polynomial} is a direct 
consequence of \eqref{compute-Hilbert-series}.
\end{rem}

\section{Expansions and contractions of monomials}
\label{sec:exp}

Throughout the remainder of this note,  $I \subsetneqq R = K[x_0,
\ldots,x_n]$ always denotes a saturated strongly stable ideal and 
$G(I)$ the 
set of its minimal monomial generators. If $n \leq 1$, then these 
ideals are principal. Thus, it is harmless to assume $n \geq 2$. 
At times, we will abuse terminology by saying that $I$ has Hilbert 
polynomial $p$ if $p$ is actually the Hilbert polynomial of the 
quotient $R/I$.

We first define left-shifts and right-shifts for monomials, and 
then use left-shifts and right-shifts to define contractions and 
expansions of monomials. We adapt Reeves's definitions for 
left-shifts and right-shifts of monomials and for contractions of 
monomials (see \cite{AR} and Remarks \ref{rem-left-right-shift}(iii) 
and \ref{rem-contr} below). Expansions will play a central role 
in the algorithm to compute all saturated strongly stable ideals 
to a given Hilbert polynomial.

\begin{defn} \label{right-left-shifts}
Let $x^{A} \in R$ be a monomial of positive degree.
\begin{itemize}
\item[(i)] The set of \emph{right-shifts} of $x^A$ is 
$$\Right(x^{A}):=\left\{x^A \frac{x_{i+1}}{x_i}:x_i|x^A, 
0 \leq i < n-1 \right\}.$$
\item[(ii)] The set of \emph{left-shifts} of $x^A$ is 
$$\Left(x^{A}):=\left\{x^A \frac{x_{i-1}}{x_i}:x_i|x^A, 
0 < i \leq n-1 \right\}.$$
\end{itemize}
\end{defn}

\begin{ex} \label{ex-right-left-shift}
Consider the monomial $x_1^2x_3 \in K[x_0,\ldots,x_5]$. As  its 
right-shifts we get
$$ \Right(x_1^2x_3) = \left\{x_1x_2x_3, x_1^2x_4 \right\}. $$
For its left-shifts we obtain
$$ \Left(x_1^2x_3) = \left\{x_0x_1x_3, x_1^2x_2 \right\}. $$
\end{ex}

\begin{rem} \label{rem-left-right-shift}
\begin{itemize}
\item[(i)] Observe that all monomials in $\Left(x^{A})$ and 
$\Right(x^{A})$ have the same degree as $x^{A}$. Furthermore, 
every monomial in $\Left(x^A)$ is larger than $x^A$ in the 
lexicographic order, and every monomial in $\Right(x^A)$ is less 
than $x^A$. In particular, $\Left(x^{A})\cap\Right(x^{A}) = 
\varnothing$ and neither of the sets, $\Left(x^{A})$ nor 
$\Right(x^{A})$, contains the monomial $x^{A}$ itself.
\item[(ii)] The set of left-shifts of any monomial of the form 
$x_0^k$ is empty ($\Left(x_0^k) = \varnothing$). This fact will 
be important below.
\item[(iii)] The original definitions for left-shifts and 
right-shifts in \cite{AR} included redundant monomials. The above 
definitions provide the smallest sets which can be used to 
determine whether an ideal will continue to be strongly stable 
after adding or removing minimal monomial generators (see Lemma 
\ref{contr-expan-stable-ideals}).
\end{itemize}
\end{rem}

Next, we introduce expansion and contractions.

\begin{defn} \label{expan-contr}
Let $x^A$ be a monomial of $R$.
\begin{itemize}
\item[(i)] If $x^A \neq 1$ is a minimal generator of $I$ such 
that $G(I) \cap \Right(x^A) = \varnothing$, then we call $x^A$ 
\emph{expandable in} $I$ (or simply \emph{expandable} if the 
ideal is understood). The \emph{expansion of $x^A$ in} $I$ is 
defined to be the ideal $I^{\E}$ generated by the set 
$$G(I^{\E}):=\left(G(I)\setminus\left\{x^A\right\}\right)\cup
\left\{x^A\cdot x_r,x^A\cdot x_{r+1},\ldots,x^A\cdot x_{n-1}\right\},$$ 
where $r = \max (x^A)$.

If $I = R$ and $x^A = 1$, then we set $I^{\E} := (x_0,\ldots,x_{n-1})$.

\item[(ii)] If $x^A \neq 1$ is a monomial in $R$ such that $x^A 
\cdot x_{n-1} \in G(I)$ (so $x^A \notin I$) and $\Left(x^A) 
\subset I$, then we call $x^A$ \emph{contractible in} $I$ (or 
simply \emph{contractible} if the ideal is understood). The 
\emph{contraction of $x^A$ in} $I$ is defined to be the ideal 
$I^{\C}$ generated by the set 
$$G(I^{\C}):=\left(G(I)\cup\left\{x^A\right\}\right)\setminus
\left\{x^A\cdot x_r,x^A\cdot x_{r+1},\ldots,x^A\cdot x_{n-1}\right\},$$
where $r = \max (x^A)$.

If $x_{n-1} \in G(I)$ and $x^A = 1$, then we set $I^{\C}:=(1)=R$.
\end{itemize}
\end{defn}

We note  that expandable monomials have been studied elsewhere.

\begin{rem} \label{Borel-generators}
The expandable monomials of a strongly stable ideal are exactly 
the Borel generators; compare our Definition \ref{expan-contr}(i) 
with Proposition 2.13 in \cite{FMS}.
\end{rem}

\begin{ex} \label{ex-expan}
Consider the saturated strongly stable ideal $I:= (x_0^3,x_0^2x_1,
x_0^2x_2) \subset K[x_0,x_1,x_2,x_3]$. The monomial $x_0^2x_2$ is 
expandable in $I$ because the monomial in $\Right(x_0^2x_2) = 
\{x_0x_1x_2\}$ is not a minimal generator of $I$. The expansion 
of $x_0^2x_2$ in $I$ is generated by 
$$G(I^{\E})=G(I)\setminus\{x_0^2x_2\}\cup\{x_0^2x_2^2\}=
\{x_0^3,x_0^2x_1,x_0^2x_2^2\}.$$

Now the monomial $x_0^2x_2$ is contractible in $I^{\E} = (x_0^3, 
x_0^2x_1, x_0^2x_2^2)$ since it is not contained in $I^{\E}$ and 
$\Left(x_0^2x_2) = \{x_0^2x_1\}$ is in $I^{\E}$. The contraction 
of $x_0^2x_2$ in $I^{\E}$ is the ideal $I$ we started with.

Similarly, the monomial $x_0^2$ is contractible in $I$ because it 
is not in the ideal, the monomial $x_0^2x_2$ is a minimal generator 
of $I$, and $\Left(x_0^2) = \varnothing \subset I$. The contraction 
of $x_0^2$ in $I$ is generated by 
$$G(I^{\C})= G(I)\cup\{x_0^2\}\setminus\{x_0^3,x_0^2x_1,x_0^2x_2\}=
\{x_0^2\}.$$

Now the monomial $x_0^2$ is expandable in $I^{\C} = (x_0^2)$ since 
it is the only minimal generator (so the set of right-shifts is 
automatically disjoint from the set of minimal generators of the 
ideal). The expansion of $x_0^2$ in $I^{\C}$ is the ideal $I$ we 
started with.
\end{ex}

As seen in this example, the contraction and expansion of a monomial 
in a saturated strongly stable ideal are inverse operations. This 
will be a useful fact.

\begin{lem} \label{expan-contr-inverse-ops}
Let $x^A \in R$ be a monomial.
\begin{itemize}
\item[(a)] If $x^A$ is expandable in $I$, then $x^A$ is contractible 
in the resulting expansion $I^{\E}$. The contraction of $x^A$ in 
$I^{\E}$ is $I$.
\item[(b)] If $x^A$ is contractible in $I$, then $x^A$ is expandable 
in the resulting contraction $I^{\C}$. The expansion of $x^A$ in 
$I^{\C}$ is $I$.
\end{itemize}
\end{lem}

\begin{proof}
These observations follow directly from Definition \ref{expan-contr}.
\end{proof}

\begin{rem} \label{rem-contr}
Following \cite{KG}, our Definition \ref{expan-contr}(ii) differs 
from Reeves's original definition in Appendix A.2 of \cite{AR} in two 
places as we insist on $x^A \cdot x_{n-1} \in G(I)$, but require only 
$\Left(x^A) \subset I$ instead of $\Left(x^A) \subset G(I)$. The first 
change is necessary for Lemma \ref{expan-contr-inverse-ops}(ii); the 
second is essential to establish Lemma \ref{contr-to-double-sat} (see 
also Example \ref{ex-contr}).
\end{rem}

Contractions and expansions are defined so that they will produce 
saturated strongly stable ideals. The proof is straightforward, but 
is included nonetheless.

\begin{lem} \label{contr-expan-stable-ideals}
If a monomial $x^A$ is contractible or expandable in $I$, then $I^{\C}$ 
or $I^{\E}$ is a saturated strongly stable ideal, respectively.
\end{lem}

\begin{proof}
Note that if $I$ is saturated, then $I^{\C}$ or $I^{\E}$ will by 
definition also be saturated.

Suppose that $x^A$ is contractible. By Remark \ref{double-sat}(i), we 
need only show that $(x_i/x_j) \cdot x^A \in I^{\C}$ for all $j$ such 
that $x_j|x^A$ and all $i<j$. Since $x^A$ is contractible, $\Left(x^A)
\subset I$. Thus, for all $j$ such that $x_j|x^A$, each monomial 
$(x_{j-1}/x_j) \cdot x^A \in I$ so the monomial is also in $I^{\C}$. 
Because $I$ is strongly stable, if $(x_{j-1}/x_j) \cdot x^A \in I$, 
then $(x_i/x_j) \cdot x^A \in I$ for all $i<j$, so $(x_i/x_j) \cdot 
x^A \in I^{\C}$ for all $i<j$.

Suppose $x^A$ is expandable. Now, we need to establish that we have a 
strongly stable ideal after removing the monomial $x^A$ from $G(I)$. 
Consider a monomial $x^B$ of the form $(x_k/x_j) \cdot x^A$ for some 
$j$ such   that $x_j|x^A$ and $k>j$. Then the monomial $x^B$ is not in 
$I$, because $(x_{j+1}/x_j) \in \Right(x^A)$, $\Right(x^A)$ is disjoint 
from $I$, and $I$ is strongly stable. Thus, the monomial $x^A$ can be 
removed from $G(I)$ without destroying strong stability.
\end{proof}

In any saturated strongly stable ideal, there will always be 
expandable monomials. If the ideal is not doubly saturated, there will 
be contractible monomials. The particular expansions and contractions 
described in the following result form the basis for Section 
\ref{almost-lex-section}.

\begin{lem} \label{exist-some-expan-contr}
In any fixed degree, the minimal monomial generator of $I$, which is 
smallest according to the lexicographic order, will be expandable.

If the ideal $I$ is not doubly saturated, then some minimal monomial 
generators will contain the variable $x_{n-1}$. In any fixed degree $d$, 
among the monomials $x^A$ of degree $d - 1$ such that $x^A x_{n-1}$ is 
a minimal monomial generator of the ideal, the monomial, which is 
largest according to the lexicographic order, will be contractible.
\end{lem}

\begin{proof}
These observations follow directly from Definition \ref{expan-contr} 
and Remark \ref{rem-left-right-shift}(i).
\end{proof}

Our aim is to use expansions to produce saturated strongly stable 
ideals from simpler ideals--ideals with fewer minimal generators or 
minimal generators of smaller degree.  We start with the following 
result, which appears as Lemma 23 in \cite{AR}. We follow Reeves's 
argument with some suitable modifications.

\begin{lem} \label{contr-to-double-sat}
There is a finite sequence of contractions taking the ideal $I$ to 
its double saturation $\sat_{x_{n-1},x_n}(I)$.
\end{lem}

\begin{proof}
Since $I$ is saturated, no minimal generators are divisible  by $x_n$. 
Consider the set $M$ of monomials in $G(I)$ that are divisible by 
$x_{n-1}$. If $M=\varnothing$, then $I$ is doubly saturated. Otherwise, 
choose the monomial $x^A \cdot x_{n-1}$ of least  degree in $M$, which 
is largest with respect to the lexicographic order $>_{lex}$. As noted 
in Lemma \ref{exist-some-expan-contr}, 
$x^{A}$ is contractible in $I$.

Let $I^{\C}$ be the contraction of $x^A$ in $I$. Note that contracting 
$x^{A}$ replaces $x^{A} \cdot x_{n-1}$ (and possibly other monomials) 
by $x^{A}$. Thus, $I^{\C}$ has the same double saturation as $I$. After 
repeating the above step some finite number of times, we get an ideal 
whose minimal generators are not divisible by $x_{n-1}$. This is the 
double saturation of $I$.
\end{proof}

\begin{ex} \label{ex-contr}
We illustrate the last proof with the ideal $I=(x_0, x_1^2, x_1x_2^3)$ 
in the ring $K[x_0,x_1,x_2,x_3]$.
\begin{itemize}
\item First we contract the monomial $x_1x_2^2$ in $I$. (Note that 
$\Left(x_1x_2^2)=\{x_0 x_2^2, x_1^2 x_2\}$ is {\em not} a subset of 
the set of minimal generators of $I$. This shows that our modification 
of Reeves's definition of contraction in \cite{AR} is needed in the 
above argument.) The resulting ideal $I_1$ is generated minimally by
$$G(I_1)=G(I)\cup\{x_1x_2^2\}\backslash\{x_1x_2^3\}=
\{x_0,x_1^2,x_1x_2^2\}.$$
\item Next, we contract $x_1x_2^2$ and  get the ideal
$$ I_2 = (x_0, x_1^2, x_1x_2). $$
\item In the last step, contracting $x_1$ in $I_2$ gives the double 
saturation
$$ I_3 = (x_0,x_1) = \sat_{x_2,x_3}(I). $$
\end{itemize}
\end{ex}

We now make the contractions necessary to get to the double saturation 
more explicit.

\begin{rem} \label{rem-contr-scheme}
Assume that the ideal $I$ is different from its double saturation. 
List the minimal generators of $I$ that are divisible by $x_{n-1}$,
$$ x^{A_1} x_{n-1}^{e_1},\; x^{A_2} x_{n-1}^{e_2},\; \ldots,\; 
x^{A_s} x_{n-1}^{e_s}, $$
where $x^{A_i}$ is not divisible by $x_{n-1}$, so that $\deg x^{A_i} 
x_{n-1}^{e_i} \leq \deg x^{A_{i+1}} x_{n-1}^{e_{i+1}}$, and in case 
of equality $x^{A_i} x_{n-1}^{e_i} >_{lex} x^{A_{i+1}} 
x_{n-1}^{e_{i+1}}$. Then the contractions in the algorithm given in 
the proof of Lemma \ref{contr-to-double-sat} use the following 
monomials
$$x^{A_1} x_{n-1}^{e_1 - 1},\; x^{A_1} x_{n-1}^{e_1 - 2}, \ldots, 
x^{A_1},\; x^{A_2} x_{n-1}^{e_2-1},\; \ldots,\; x^{A_2}, \ldots, 
x^{A_s}$$
in the stated order. Thus, we need $e_1+e_2+\ldots+e_s$ contractions 
to compute the double saturation of $I$.
\end{rem}

Since this process is reversible, we can recover an ideal from its 
double saturation:

\begin{cor} \label{expan-from-double-sat}
There is a finite sequence of expansions taking the double 
saturation of an ideal $\sat_{x_{n-1},x_n}(I)$ to the ideal $I$. In 
particular, the necessary number of expansions can be determined by 
adding up the exponents of $x_{n-1}$ in the minimal generators of $I$.
\end{cor}

\begin{proof}
The sequence of contractions described in  Remark 
\ref{rem-contr-scheme}, which take $I$ to its double saturation, can 
be reversed and considered as expansions by Lemma 
\ref{expan-contr-inverse-ops}.
\end{proof}

We conclude this section by describing the change of the Hilbert 
function under contraction or expansion.

\begin{lem} \label{change-of-Hilbert-funct}
\begin{itemize}
\item[(a)] Let $I^{\E}$ be the expansion of $x^A$ in $I$. Then 
$$ h_{R/I^{\E}} (j) = \left \{ 
\begin{array}{cl} 
h_{R/I} (j) & \mif j < \deg (x^A) \\ 
h_{R/I} (j) + 1 & \mif j \geq \deg (x^A)
    \end{array} \right. . $$
\item[(b)] Let $I^{\C}$ be the contraction of $x^B$ in $I$. Then 
$$ h_{R/I^{\C}} (j) = \left \{ 
\begin{array}{cl} h_{R/I} (j) & \mif j < \deg (x^B) \\ 
h_{R/I} (j) - 1 & \mif j \geq \deg (x^B)
    \end{array} \right. . $$
\end{itemize}
\end{lem}

\begin{proof}
(a) We have $I^{\E} \subset I$. Furthermore, if $j \geq \deg (x^A)$, 
then $x^A \cdot x_n^{j - \deg (x^A)}$ is the only monomial in $[I]_j 
\setminus [I^{\E}]_j$. The claim follows.

(b) Now, $I \subset I^{\C}$, and $x^B \cdot x_n^{j - \deg (x^B)}$ is 
the only monomial in $[I^{\C}]_j \setminus [I]_j$, provided $j \geq 
\deg (x^B)$.
\end{proof}

We can now determine the number of expansions to recover an ideal 
from its double saturation in a more abstract manner.

\begin{cor} \label{number-expan}
The number of expansions needed to take $J = \sat_{x_{n-1},x_n}(I)$ 
to $I$ is $$ p_{R/I} - p_{R/J}. $$
\end{cor}

\section{Strongly stable ideals with a given Hilbert polynomial}

In this section, we describe how to produce {\em all} saturated 
strongly stable ideals with a given Hilbert polynomial. We develop 
a few more tools, which culminate in Theorem \ref{link-all-ideals} 
and Algorithm \ref{alg-stable-Hilbpoly}. We start with the simplest 
case, ideals with constant Hilbert polynomial:

\begin{lem} \label{constant-Hilbert-polynomial}
Let $I \subset R$ be a saturated strongly stable ideal with 
constant Hilbert polynomial, say $p_{R/I} = c$. Then 
$\sat_{x_{n-1},x_n}(I) = (1) = R$. Moreover, any saturated strongly 
stable ideal $J \subset R$ with $p_{R/J} = c$ can be obtained from 
the ideal $(1)$ using $c$ suitable expansions.
\end{lem}

\begin{proof}
If $x_{n-1}^k \in I$, then $\sat_{x_{n-1},x_n}(I) = (1) = R$ by 
Remark \ref{double-sat}(iii). Assume that no power of $x_{n-1}$ is 
in $I$. Let $j$ be any positive integer. Since $I$ is strongly stable, 
no monomial of the form $x_{n-1}^{j-i} \cdot x_n^{i} \in I$ for 
$0 \leq i \leq j$. Hence, there are at least $j+1$ monomials not 
contained in $[I]_j$ for every $j>0$, which contradicts 
$p_{R/I}(z)=c$. Thus, some power of $x_{n-1}$ is in $I$. The final 
claim is now a consequence of Corollaries 
\ref{expan-from-double-sat} and \ref{number-expan}.
\end{proof}

Recall some previously introduced notation:  $R^{(j)} := K[x_0, 
\ldots, x_{n-j}]$ is the polynomial ring where the last $j$ 
variables of $R$ have been removed. If $I \subset R$ is a saturated 
strongly stable ideal with Hilbert polynomial $p$, then the 
restriction of its double saturation $\sat_{x_{n-1},x_n}(I)$ to 
$R^{(1)} := K[x_0,\ldots,x_{n-1}]$ is a saturated strongly stable 
ideal in $R^{(1)}$ with a Hilbert polynomial that can be computed 
from $p$:

\begin{lem}\label{Delta-p}
If $I$ is a saturated strongly stable ideal with Hilbert  polynomial 
$p(z)$ and double saturation $J = \sat_{x_{n-1},x_n}(I)$, then the 
Hilbert polynomial of $J^{(1)} := J \cdot R^{(1)} \subset R^{(1)}$ is 
$p_{R^{(1)}/J^{(1)}} (z) = \Delta p(z) := p(z) - p(z-1)$.
\end{lem}

\begin{proof}
Setting $I^{(1)} = I \cdot R^{(1)}$, multiplication by $x_n$ induces 
the exact sequence
$$\begin{CD} 0 @>>> R/I(-1) @>x_n>> R/I @>>> R^{(1)}/I^{(1)} @>>> 0, 
\end{CD}$$
since $x_n$ is not a zero divisor of $R/I$. Now, 
$p_{R^{(1)}/I^{(1)}}(z) = \Delta p(z).$ Passing to $J^{(1)}$, the 
saturation of $I^{(1)}$, does not change the Hilbert polynomial, so 
$p_{R^{(1)}/J^{(1)}} (z) = \Delta p(z)$.
\end{proof}

This result can be extended. If $p(z)$ is a Hilbert polynomial of 
degree $d$, we set  $\Delta^0 p(z):= p(z)$, and recursively define 
$\Delta^j p(z):= \Delta^{j-1} p(z) - \Delta^{j-1} p(z-1)$ for 
$1 \leq j \leq d$. Thus, $\Delta = \Delta^1$. Now, if $I$ is a 
saturated strongly stable ideal, then, for $0 \leq j \leq d$, we denote 
by $I^{(j)} \subset R^{(j)}$, the saturated strongly stable ideal whose 
generating set is obtained by setting $x_{n-j} = \ldots = x_{n-1} = 1$ 
in the monomial generators of $I$. Note that the ideal $I^{(j+1)} \cdot 
R^{(j)}$ is the double saturation of $I^{(j)}$. Repeating the argument 
in Lemma \ref{Delta-p} shows that $\Delta^j p(z)$ is the Hilbert 
polynomial of the ideal $I^{(j)}$:

\begin{cor} \label{Delta-j-p}
If $I$ is a saturated strongly stable ideal with Hilbert polynomial 
$p(z)$ of degree $d$, and $I^{(j)} \subset R^{(j)}$ is the ideal 
obtained by setting $x_{n-j} = \ldots = x_{n-1} = 1$ in the monomial 
generators of $I$, then the Hilbert polynomial of $I^{(j)}$ is 
$p_{R^{(j)}/I^{(j)}} (z) = \Delta^j p(z)$ for $0 \leq j \leq d$.
\end{cor}

We are now ready to prove the main result of this section.

\begin{thm} \label{link-all-ideals}
Let $I \subsetneqq R$ be a saturated strongly stable ideal with Hilbert 
polynomial $p(z)$ of degree $d$. Then there is a finite sequence of 
expansions (in the appropriate rings) that take the ideal $(1) = 
R^{(d)}$ to the ideal $I \subset R$.

In particular, the number of expansions needed in $R^{(j)}$ to take 
$I^{(j+1)} \cdot R^{(j)}$ to $I^{(j)}$ is $\Delta^{j} p(z) - 
p_{R^{(j)}/I^{(j+1)} R^{(j)}}(z)$, which, in the notation of Theorem 
\ref{L_p}, is at most $a_j$, for $j = 0,\ldots,d$. The total number 
of expansions needed to take $(1) = R^{(d)}$ to the ideal $I \subset 
R$ is at most $b_0$.
\end{thm}

\begin{proof}
Let $p(z)$ be the Hilbert polynomial of $R/I$. We induct on the degree 
of $p(z)$. If $\deg p = 0$, then we are done by Lemma 
\ref{constant-Hilbert-polynomial}. Assume $\deg p > 0$. Since 
$\deg \Delta^1 p = \deg p - 1$, we conclude by the induction 
hypothesis that there is a finite sequence of expansions that takes 
the ideal $(1) \subset R^{(d)}$ to $J^{(1)} = \sat_{x_{n-1},x_n}(I) 
\subset R^{(1)}$, the double saturation of $I$ as an ideal in 
$R^{(1)}$. Considering the corresponding extension ideal in $R$, the 
ideal $I$ can be obtained from $J^{(1)} \cdot R$ by Corollary 
\ref{expan-from-double-sat} using a finite number of expansions.

The claim that the number of expansions needed in the ring $R^{(j)}$ 
to take the ideal $I^{(j+1)} \cdot R^{(j)}$ to the ideal $I^{(j)}$ 
is $\Delta^{j} p(z) - p_{R^{(j)}/I^{(j+1)} R^{(j)}}(z)$ follows from 
Corollaries \ref{number-expan} and \ref{Delta-j-p}. Thus it remains 
to show that
\begin{equation} \label{eq:steps}
 \Delta^{j} p(z) - p_{R^{(j)}/I^{(j+1)} R^{(j)}}(z) \leq a_j
\end{equation}
because the final assertion then follows by recalling that  
$b_0 = a_0 + \cdots + a_d$.

In order to establish Inequality \eqref{eq:steps} write  the given 
Hilbert polynomial as in Equation \eqref{formula-Hilbert-polynomial} as
\begin{equation*}
p(z) = \sum_{i=0}^{d} \left[\binom{z+i}{i+1}-\binom{z+i-b_i}{i+1}\right].
\end{equation*}
By Lemma \ref{Delta-p}, the Hilbert polynomial of $I^{(j)}$ is
\begin{eqnarray*}
  p_{R^{(j)}/I^{(j)}} (z) & = & \Delta^j p (z) \\
  & = & \sum_{i=j}^{d}\left[\binom{z+i-j}{i+1-j}
  -\binom{z+i-b_i-j}{i+1-j}\right].
\end{eqnarray*}
Using Theorem \ref{L_p}, it follows that exactly $a_j$ expansions in 
$R^{(j)}$ are needed to take the lexicographic ideal $L_{\Delta^{j+1} 
p} R^{(j)}$  to the lexicographic ideal $L_{\Delta^{j} p}$ of 
$R^{(j)}$. Since $R^{(j+1)}/I^{(j+1)}$ and $R^{(j+1)}/L_{\Delta^{j+1} 
p}$ have the same Hilbert polynomial, namely $\Delta^{j+1} p$, 
Inequality \eqref{eq:steps} is equivalent to
\begin{equation} \label{Hilb-poly-ineq}
p_{R^{(j)}/I^{(j+1)} R^{(j)}}(z) \geq p_{R^{(j)}/L_{\Delta^{j+1} p} 
R^{(j)}}(z).
\end{equation}
(The difference of the two polynomials is a constant.) However,  
the latter estimate is a consequence of Proposition 
\ref{prop-smallest-Hilb}(b) because 
$L_{\Delta^{j+1}p}\subset R^{(j+1)}$ is the saturation of 
$L_{\Delta^{j} p} R^{(j+1)}$ in $R^{(j+1)}$, so, for all integers $k$,
$$h_{R^{(j+1)}/I^{(j+1)}}(k) \geq h_{R^{(j+1)}/L_{\Delta^{j+1} p}}(k).$$
Summing over $k$ on both sides of this inequality, we get the Hilbert 
functions of $R^{(j)}/I^{(j+1)} R^{(j)}$ and $R^{(j)}/L_{\Delta^{j+1} 
p} R^{(j)}$, respectively. Now, Inequality \ref{Hilb-poly-ineq} follows.
\end{proof}

Note that the estimate on the number of needed expansions is sharp. 
This follows from Lemma \ref{almost-lex-double-sat} below.

The particular expansions leading to the lexicographic ideal $L_p$ 
can be made explicit.

\begin{rem}\label{rem-expan-lex-ideal}
In Theorem \ref{link-all-ideals}, the lexicographic ideal will be 
obtained if, at each step, the minimal monomial generator to be 
expanded is of the highest degree, and is smallest according to the 
lexicographic order in that degree. This follows by Proposition 
\ref{prop-smallest-Hilb}(b) and Lemma  \ref{exist-some-expan-contr}.
\end{rem}

Using Theorem \ref{link-all-ideals} and its proof, we can now  give 
the desired algorithm to compute all saturated strongly stable ideals 
with a prescribed Hilbert polynomial.

\begin{alg} \label{alg-stable-Hilbpoly} (\emph{Generating all 
saturated strongly stable ideals with a given Hilbert polynomial})
Let $p(z)$ be a nonzero Hilbert polynomial of degree $d$ of a 
graded quotient of $R$.
\begin{enumerate}
\item Compute the polynomials $\Delta^1 p(z)$, $\Delta^2 p(z)$, 
$\ldots$, $\Delta^d p(z)$. (Note that $\Delta^d p(z) = c$ for some 
$c \in \mathbb{N}$.)
    Set $\Set^{(d)} = \cdots = \Set^{(0)} = \varnothing$.
\item Generate $\Set^{(d)}$, the set of all saturated strongly 
stable ideals $I$ in $R^{(d)}$ with Hilbert polynomial 
$p_{R^{(d)}/I} (z) = \Delta^d p (z) = c$, using $c$ successive 
expansions of monomial generators starting with the ideal 
$(1) = R^{(d)}$. Exhaust all choices for $c$ successive expansions.
\item For $j = d-1, \; d-2, \; \ldots, \; 0$, repeat the following 
steps for each ideal $I \in \Set^{(j+1)}$: \\ Compute 
$p_{R^{(j)}/I}(z)$ (using Equation \eqref{compute-Hilbert-polynomial}). 
Let $a=\Delta^j p(z)-p_{R^{(j)}/I}(z)$.
  \begin{itemize}
  \item If $a \geq 0$, then perform $a$ successive expansions of 
  monomial generators of $I$ to obtain ideals with Hilbert 
  polynomial $\Delta^{j} p(z)$. Exhaust all choices for $a$ 
  successive expansions. Add these ideals to $\Set^{(j)}$.
  \item If $a<0$, then continue with the next ideal $I$ in 
  $\Set^{(j+1)}$.
  \end{itemize}
\item Return the set $\Set^{(0)}$.
\end{enumerate}
\end{alg}

\begin{proof} (\emph{Correctness})
By Theorem \ref{link-all-ideals}, every saturated strongly stable 
ideal with Hilbert polynomial $p(z)$ will be generated by this 
algorithm, as long as every possible sequence of expansions is 
carried out at each step. Also, every ideal generated by this 
process will be saturated and strongly stable and have the desired 
Hilbert polynomial.

The algorithm terminates for any given Hilbert polynomial, since 
the number of steps performed in (3) is bounded by the degree of 
the Hilbert polynomial and the number of generators in each ideal 
computed in each loop is finite.
\end{proof}

Note that different algorithms to achieve the same goal have been 
proposed by Reeves in \cite{AR} and Cioffi, Lella, Marinari, and 
Roggero in \cite{CLMR}. We defer a comparison of these algorithms 
to Remark \ref{comparison-of-algorithms}.

When carrying out Algorithm \ref{alg-stable-Hilbpoly}, one can order 
the expansions so that each ideal is produced in a unique way.

\begin{rem} \label{make-expan-unique}
One natural ordering of minimal generators is to always list the 
monomials first by degree in increasing order and then 
lexicographically in each degree. When expanding in some ring 
$R^{(j)}$, always pick monomials, which precede all other monomials 
that have been expanded in this ring (those monomials divisible by 
the variable $x_{n-j-1}$). (Thus, the expanded monomials, leading 
to a certain ideal, will be strictly increasing according to this 
order and, hence, unique.) This is the reverse of the order for 
contractions discussed in Remark \ref{rem-contr-scheme}.
\end{rem}

We include an example to illustrate this algorithm.

\begin{ex}\label{ex-first-alg}
Suppose we wish to find all saturated strongly stable ideals with 
Hilbert polynomial $p(z) = \frac{3}{2} z^2 + \frac{5}{2} z = 
\binom{z+2}{3}-\binom{z-1}{3}+\binom{z+1}{2}-\binom{z-3}{2}+
\binom{z}{1}-\binom{z-5}{1}$ in $R = K[x_0,x_1,x_2,x_3,x_4]$.

\begin{itemize}
\item First we compute $\Delta^1 p(z)$ and $\Delta^2 p(z)$: 
$$\Delta^1 p(z) = 3z + 1, \qquad \Delta^2 p(z) = 3$$

\item Next we generate all ideals in $R^{(2)} = K[x_0,x_1,x_2]$ 
with Hilbert polynomial $\Delta^2 p(z) = 3$ using 3 successive 
expansions and starting from $(1) = R^{(2)}$. We get two ideals: 
$$I = (x_0,x_1^3), \qquad J = (x_0^2,x_0x_1,x_1^2)$$

\item Now we generate all ideals in $R^{(1)}$ with Hilbert 
polynomial $\Delta^1 p(z) = 3z + 1$. We compute the Hilbert 
polynomials of $I$ and $J$ in $R^{(1)}$: 
$$p_{R^{(1)}/I}(z) = 3z, \qquad p_{R^{(1)}/J}(z) = 3z + 1$$ 
We perform one expansion in $I$ to obtain the following ideals: 
$$I_1=(x_0,x_1^4,x_1^3x_2), \qquad I_2=(x_0^2,x_0x_1,x_0x_2,x_1^3)$$ 
We perform no expansions in $J$ (as it already has the desired 
Hilbert polynomial).

\item Finally we generate all ideals in $R$ with Hilbert 
polynomial $p(z) = 3z^2/2 + 5z/2$. We compute the Hilbert 
polynomials of $I_1$, $I_2$, and $J$: 
$$p_{R/I_1}(z) =  3z^2/2 + 5z/2 - 1, \qquad p_{R/I_2}(z) 
= 3z^2/2 + 5z/2 + 1 $$ 
$$ p_{R/J}(z) = 3z^2/2 + 5z/2 + 1$$ 
We ignore the ideals $I_2$ and $J$ because their Hilbert 
polynomials in $R$ are too large. We perform one expansion 
in $I_1$ to obtain the following ideals: 
$$(x_0, x_1^4, x_1^3x_2^2, x_1^3x_2x_3), \qquad 
(x_0^2, x_0x_1, x_0x_2, x_0x_3, x_1^4, x_1^3x_2)$$
\end{itemize}

Thus, there are two saturated strongly stable ideals in $R$ 
with Hilbert polynomial $p(z) = 3z^2/2 + 5z/2$. Note that 
$(x_0, x_1^4, x_1^3x_2^2, x_1^3x_2x_3)$ is the lexicographic 
ideal.
\end{ex}

\section{Almost lexsegment ideals with a given Hilbert polynomial} 
\label{almost-lex-section}

In this section, we develop an algorithm for producing a unique 
ideal for each Hilbert series associated to a given Hilbert 
polynomial. This algorithm is helpful when looking for ideals 
with a fixed Hilbert polynomial, which have maximal Betti numbers.

We begin by introducing the class of strongly stable ideals in 
which we are now interested. If a strongly stable ideal is 
saturated, then no minimal monomial generators contain the last 
variable $x_n$. Thus, the ideal can be considered in the 
polynomial ring $R^{(1)}$, where the variable $x_n$ has been 
removed. This class of ideals is characterized by the fact 
that they are lexsegment ideals when viewed in the smaller ring 
$R^{(1)}$.

\begin{defn} \label{almost-lex-ideal}
A saturated strongly stable ideal $I \subset R$ is called 
\emph{almost lexsegment} if $I \cdot R^{(1)}$ is a lexsegment ideal.
\end{defn}

\begin{ex} \label{ex-almost-lex-ideal}
Consider the saturated strongly stable ideals $I_1=(a^2,ab,ac,b^2)$, 
$I_2=(a^2,ab,ac,b^3,b^2c)$, and $I_3=(a^2,ab,ac^2,b^3,bc^2)$ in 
$R=K[a,b,c,d]$. $I_1$, $I_2$ and $I_3$ are almost lexsegment 
ideals. $I_1$ is generated by the first four monomials of 
$R^{(1)}=K[a,b,c]$ in degree two. $I_2$ contains the first three 
monomials of $R^{(1)}$ in degree two, the first seven monomials of 
$R^{(1)}$ in degree three, etc; $I_3$ contains the first two 
monomials of $R^{(1)}$ in degree two, the first eight monomials 
of $R^{(1)}$ in degree three, etc.
\end{ex}

We will now focus on characterizing how to generate almost 
lexsegment ideals. The process will be similar to the previous 
algorithm, except for two simplifications: all lexsegment ideals 
have the same double saturation and are produced by certain expansions.

Recall the lexicographic ideal $L_p$ and the nonnegative integers 
$a_i$ introduced earlier in Theorem \ref{L_p}, which are 
associated to each Hilbert polynomial.

\begin{lem} \label{almost-lex-double-sat}
Every almost lexsegment ideal with Hilbert polynomial $p(z)$ has 
the same double saturation, namely $L_{\widetilde{p}}$, where 
$$\widetilde{p} (z) = p(z) - a_0.$$
\end{lem}

\begin{proof}
Using the definition of $\widetilde{p}$, we see that the ideal 
$L_{\widetilde{p}}$ is doubly saturated by Theorem \ref{L_p} 
(because no minimal generator will be divisible by $x_{n-1}$). 
Thus, the ideal $L_{\widetilde{p}} \cdot R^{(1)} \subset R^{(1)}$ 
is the unique saturated lexsegment ideal of $R^{(1)}$ with 
Hilbert polynomial $\Delta p(z)$ by Lemma \ref{Delta-p}.

The double saturation of an almost lexsegment ideal $I \subset R$ 
with Hilbert polynomial $p(z)$ will also be a saturated lexsegment 
ideal in $R^{(1)}$ with Hilbert polynomial $\Delta p(z)$. Thus, 
the double saturation must be $L_{\widetilde{p}}$.
\end{proof}

Note that the uniqueness statement of the double saturation in 
Lemma \ref{almost-lex-double-sat} is equivalent  to  Proposition 
2.3 in \cite{CM}. The explicit description of the double 
saturation is new.

We give a name to the special expansions and contractions that 
were noted earlier in Lemma \ref{exist-some-expan-contr}.

\begin{defn} \label{lex-expan}
Let $I \subset R$ be an almost lexsegment ideal.
\begin{itemize}
\item[(i)] In any fixed degree, an expansion of the minimal 
monomial generator of $I$, which is last according to the 
lexicographic order, is called a \emph{lex expansion}.
\item[(ii)] In any fixed degree, a contraction of the monomial 
$x^A$ such that $x^A x_{n-1}$ is a minimal monomial generator, 
which is first according to the lexicographic order, is called 
a \emph{lex contraction}.
\end{itemize}
\end{defn}

Note that lex expansions and lex contractions are inverse 
operations. (For any lex expansion, there is a lex contraction 
which will undo it, and vice versa.)

Lex expansions and lex contractions are the only tools needed to 
produce almost lexsegment ideals:

\begin{lem} \label{expan-contr-almost-lex-ideals}
If $I$ is an almost lexsegment ideal, then applying a lex expansion 
or a lex contraction to $I$ will produce another almost lexsegment 
ideal.

In fact, the only expansions of almost lexsegment ideals which 
produce almost lexsegment ideals are the lex expansions, and, 
similarly, the only contractions of almost lexsegment ideals which 
produce almost lexsegment ideals are the lex contractions.
\end{lem}

\begin{proof}
Assume $I \subset R$ is an almost lexsegment ideal.

Expanding a monomial $x^A$ of degree $d$ only changes the ideal 
$I \cdot R^{(1)}$ in degree $d$ (by removing the monomial $x^A$ 
from $I$). If $x^A$ is the smallest minimal monomial generator of 
$I$ in degree $d$ according to the lexicographic order, then the 
expansion of $x^A$ will be an almost lexsegment ideal. Expanding a 
monomial of degree $d$ which comes before $x^A$ in the lexicographic 
order will produce an ideal which is not almost lexsegment.

Similarly, contracting a monomial $x^A$ of degree $d$ only changes 
the ideal $I \cdot R^{(1)}$ in degree $d$ (by adding the monomial 
$x^A$). If $x^A$ is the largest minimal monomial generator of $I$ 
in degree $d$ according to the lexicographic order, then the 
contraction of $x^A$ will be an almost lexsegment ideal. Contracting 
a monomial of degree $d$ which comes after $x^A$ in the lexicographic 
order will produce an ideal which is not almost lexsegment.
\end{proof}

We illustrate the last lemma with an example.

\begin{ex}
 Consider again the almost lexsegment ideals $I_1=(a^2,ab,ac,b^2)$, 
$I_2=(a^2,ab,ac,b^3,b^2c)$, and $I_3=(a^2,ab,ac^2,b^3,bc^2)$ in 
$R=K[a,b,c,d]$ from Example \ref{ex-almost-lex-ideal}.

Observe that the smallest monomial generator in $I_1$ of degree 
two, according to the lexicographic order, is $b^2$. This monomial 
is expandable, and expanding it produces the almost lexsegment 
ideal $I_2$. The monomial $ac$ is also expandable in $I_1$, but 
expanding it produces an ideal, $J=(a^2,ab,ac^2,b^2)$, which is 
not almost lexsegment (since $ac$ is not in $J$, $b^2$ is in 
$J$, and $ac >_{lex} b^2$).

Observe that there are two contractible monomials in $I_3$: $ac$ 
and $b^2$. As $ac$ is greater than $b^2$ in the lexicographic 
order, contracting $ac$ produces the almost lexsegment ideal 
$I_2$, while contracting $b^2$ produces the ideal $J$, which is 
not almost lexsegment.
\end{ex}

We summarize the above results:

\begin{cor} \label{lex-expan-almost-lex-ideals}
Each almost lexsegment ideal with Hilbert polynomial $p(z)$ can be 
obtained from its double saturation $L_{\widetilde{p}}$ through a 
sequence of $a_0$ lex expansions (exclusively) through almost 
lexsegment ideals.
\end{cor}

\begin{proof}
If an almost lexsegment ideal is not doubly saturated, then we can 
perform a lex contraction to produce another almost lexsegment 
ideal. Repeating a finite number of times will yield the double 
saturation. Since lex expansions and lex contractions are inverse 
operations, we can go the other direction.

The number of needed expansions is $a_0$ by Lemma 
\ref{almost-lex-double-sat} and Corollary \ref{number-expan}.
\end{proof}

Combining Corollaries \ref{number-expan} and 
\ref{lex-expan-almost-lex-ideals}  yields the following procedure.

\begin{alg} \label{alg-lex-Hilbpoly}
(\emph{Generating all almost lexsegment ideals with a given Hilbert 
polynomial}) Let $p(z)$ be a nonzero Hilbert polynomial of some 
graded quotient of $R$.
\begin{enumerate}
\item Compute $a_0$ from $p(z)$ and the double saturation of the 
lexicographic ideal, $L_{\widetilde{p}}$ (as in Theorem \ref{L_p}), 
where $\widetilde{p} (z) = p(z) - a_0$.
\item Perform $a_0$ successive lex expansions of monomial generators 
of $L_{\widetilde{p}}$. Exhaust all choices for $a_0$ successive lex 
expansions.
\end{enumerate}
\end{alg}

The following example illustrates this process.

\begin{ex}
Suppose we wish to find all almost lexsegment ideals with Hilbert 
polynomial $p(z) = 2z^2 + z + 1$ in $R=K[x_0,x_1,x_2,x_3,x_4]$.

\begin{itemize}
\item First we compute the double saturation of the lexicographic 
ideal for $p$. The lexicographic ideal is 
$(x_0, x_1^5, x_1^4x_2^2, x_1^4x_2x_3^2)$ so
$$ L_{\widetilde{p}} = (x_0, x_1^5, x_1^4x_2). $$
Note that $a_0=2$.

\item Next we make two lex expansions in all possible ways to 
produce the following four almost lexsegment ideals with the 
desired Hilbert polynomial:
$$ (x_0, x_1^5, x_1^4x_2^2, x_1^4x_2x_3^2), \qquad 
(x_0, x_1^6, x_1^5x_2, x_1^5x_3, x_1^4x_2^2, x_1^4x_2x_3) $$ 
$$ (x_0^2, x_0x_1, x_0x_2, x_0x_3, x_1^5, x_1^4x_2^2, x_1^4x_2x_3), 
\qquad (x_0^2, x_0x_1, x_0x_2, x_0x_3^2, x_1^5, x_1^4x_2). $$
\end{itemize}
\end{ex}

As noted in the introduction, Caviglia and Murai \cite{CM} recently 
showed that there is a saturated ideal which achieves maximal total 
Betti numbers among all ideals with a given Hilbert polynomial. By 
a result of Bigatti, Hulett and Pardue, it is enough to consider 
almost lexsegment ideals when looking for ideals with maximal Betti 
numbers. Using Algorithm \ref{alg-lex-Hilbpoly}, one can determine 
all such ideals.

\begin{rem} \label{rem:alg-for-max-Betti-numbers}
If one only wants to produce the almost lexsegement ideals with 
maximal Betti numbers, suitable modifications significantly reduce 
the number of ideals that are produced in Algorithm 
\ref{alg-lex-Hilbpoly}. In fact, at the beginning of the algorithm 
it is enough to repeatedly expand {\em all} monomial of least degree 
in the ideal as many times as possible. The justification for this 
modification requires very different techniques and will appear in 
a forthcoming paper.
\end{rem}

Caviglia and Murai note in their paper \cite{CM} that their proof 
``is very long and complicated'' and their construction ``is not 
easy to understand.'' Examples \ref{max-Bettis-max-HF} and 
\ref{max-Bettis-min-HF} show that there can be more than one ideal 
with maximal Betti numbers. A simpler construction or proof could 
perhaps be found by choosing a different set of ideals. This 
motivates the questions: How many ideals attain maximal Betti 
numbers and how  can they be distinguished?

One idea is to consider the Hilbert function of the ideals in 
question. Because the ideals are almost lexsegment, their Hilbert 
functions will be distinct. One might hope that among all ideals 
with maximal Betti numbers, there is one which has a Hilbert 
function which is either larger in all degrees than the other 
Hilbert functions, or which is smaller in all degrees. 
Unfortunately, the following two examples show that this is 
not the case.

Notice however, that, by a result of Valla in \cite{GV}, among 
the almost lexsegment ideals with a constant Hilbert polynomial 
and maximal Betti numbers, there is one ideal  with a maximal 
Hilbert function. Such an ideal does not exist if the Hilbert 
polynomial has positive degree.

\begin{ex}\label{max-Bettis-max-HF}
In the polynomial ring $K[x_0,x_1,x_2,x_3,x_4]$, there are 509 
saturated strong\-ly stable ideals with Hilbert polynomial 
$p(z) = z^2 + 5z + 3$. Of these, 129 are almost lexsegment ideals, 
and four ideals attain maximal Betti numbers. All four ideals are 
obtained by making two lex expansions in the ideal
$$ (x_0^3, x_0^2x_1, x_0^2x_2,
x_0^2x_3, x_0x_1^2, x_0x_1x_2, x_0x_1x_3, x_0x_2^2, x_0x_2x_3,
x_0x_3^2, x_1^4, x_1^3x_2, x_1^3x_3, x_1^2x_2^3). $$

To maximize the Hilbert function, we want to expand in the smallest 
degree possible, but we have two choices: either we expand $x_0x_3^2$ 
and $x_1^2x_2^3$ (to maximize the Hilbert function in degree three) 
to obtain the ideal
$$ (x_0^3, x_0^2x_1, x_0^2x_2, x_0^2x_3, x_0x_1^2, x_0x_1x_2, 
x_0x_1x_3,x_0x_2^2, x_0x_2x_3, x_0x_3^3, x_1^4, x_1^3x_2, 
x_1^3x_3, x_1^2x_2^4, x_1^2x_2^3x_3), $$
or we expand $x_1^3x_2$ and $x_1^3x_3$ (to maximize the Hilbert 
function in degree four) to obtain
$$ (x_0^3, x_0^2x_1, x_0^2x_2, x_0^2x_3, x_0x_1^2, x_0x_1x_2, 
x_0x_1x_3, x_0x_2^2, x_0x_2x_3, x_0x_3^2, x_1^4, x_1^3x_2^2, 
x_1^3x_2x_3, x_1^3x_3^2, x_1^2x_2^3). $$
The Hilbert functions of these two ideals are incomparable.
\end{ex}

Minimal Hilbert functions among the ideals with maximal Betti 
numbers do not exist even in the case of a constant Hilbert 
polynomial.

\begin{ex}\label{max-Bettis-min-HF}
In the polynomial ring $K[x_0,x_1,x_2,x_3]$, there are 6,481 
saturated strongly stable ideals with Hilbert polynomial $p(z) 
= 31$. Of these, 2,649 are almost lexsegment ideals, and five 
ideals attain maximal Betti numbers. All five ideals are 
obtained by making eleven lex expansions in the ideal
$$ (x_0, x_1, x_2)^4. $$

To minimize the Hilbert function, we want to expand in the largest 
degree possible, but we have two choices: either we expand the last 
nine monomials in degree four and expand the last monomial in the 
largest degree twice more (to minimize the Hilbert function in 
degree four) to obtain the ideal
\begin{eqnarray*}
  \lefteqn{(x_0^4, x_0^3x_1, x_0^3x_2, x_0^2x_1^2, x_0^2x_1x_2, 
x_0^2x_2^2, x_0x_1^4, x_0x_1^3x_2, x_0x_1^2x_2^2, x_0x_1x_2^3,} 
  \hspace*{6cm} \\
  && x_0x_2^4, x_1^5, x_1^4x_2, x_1^3x_2^2, x_1^2x_2^3, x_1x_2^4, x_2^7),
\end{eqnarray*}
or we expand the last six monomials in degree four and the last 
five monomials in degree six (to minimize the Hilbert function in 
degree five) to obtain
\begin{eqnarray*}
  \lefteqn{(x_0^4, x_0^3x_1, x_0^3x_2, x_0^2x_1^2, x_0^2x_1x_2, 
x_0^2x_2^2, x_0x_1^3, x_0x_1^2x_2, x_0x_1x_2^2, x_0x_2^3,} 
  \hspace*{6cm} \\
  && x_1^6, x_1^5x_2, x_1^4x_2^2, x_1^3x_2^3, x_1^2x_2^4, x_1x_2^5, x_2^6).
\end{eqnarray*}
The Hilbert functions of these two ideals are incomparable.
\end{ex}

\section{Strongly stable ideals with a given Hilbert series}

We now present an algorithm for producing all saturated strongly 
stable ideals with a fixed Hilbert series. This process is similar 
to the procedure for producing the lexsegment ideal for a 
prescribed Hilbert series. In that procedure, one simply adds 
monomial generators, in the appropriate degree, according to the 
lexicographic order until the desired Hilbert series is obtained. 
We adapt this strategy by adding any monomial generator, in the 
appropriate degree, which yields another saturated strongly stable 
ideal. However, we make several observations to simplify this 
process and to make it easier to implement.

Monomial generators will be added to an ideal in order of 
increasing degree: generators in lowest degree will be added first, 
starting with a power of the variable $x_0$ (because if a principal 
ideal is strongly stable, it must be generated by a power of $x_0$) 
and ending with the generators of highest degree.

For each saturated strongly stable ideal $I$, we maintain a list, 
$L_I$, of the monomials which can be added to the generators of $I$, 
so that the resulting ideal is strongly stable. We also record the 
``remaining portion'' of the numerator of the Hilbert series, 
$f_I(t) = \sum_{i=0}^r C_it^i$, using Equation 
\ref{compute-Hilbert-series}. We always add monomials of degree 
$sd(f_I) = \min\{i:C_i \neq 0\}$, the smallest degree for which 
there is a non-zero coefficient in $f_I(t)$. Recall our notation 
$l_A = \max (x^A)$ for the max index of the monomial $x^A$.

To ensure that each saturated strongly stable ideal is created in 
a unique way, monomial generators are added lexicographically.

\begin{alg} \label{alg-Hilbert-series} (\emph{Computation of all 
saturated strongly stable ideals with a given Hilbert series}) 
Let $g(t)$ be the numerator of the non-reduced Hilbert series of 
a graded quotient of  $R$.
\begin{enumerate}
\item Set $\Set = \M = \varnothing$. Compute $f_{(0)}(t)=1-g(t)$ 
and $sd(f_{(0)})$. Add the ideal $I=(x_0^{sd(f_{(0)})})$ to $\M$. 
Update $f_I(t)$ to $f_{(0)}(t)-t^{sd(f_{(0)})}$, compute $sd(f_I)$, 
and set $L_I$ to $\{x_0^{sd(f_{(0)}) - 1} x_1^{sd(f_I)-sd(f_{(0)})+1}\}$.
\item Repeat until $\M$ is empty. Choose an ideal $I \in \M$. 
Do one of the following:
  \begin{itemize}
  \item If $f_I(t)=0$, remove the ideal $I$ from $\M$ and add it 
  to $\Set$.
  \item If $L_I=\varnothing$, remove $I$ from $\M$ and continue 
  with the next ideal in $\M$.
  \item If $f_I(t) \neq 0$ and $L_I \neq \varnothing$, remove $I$ 
  from $\M$ and replace it with the $|L_I|$ ideals obtained by 
  adding a single monomial $x^B$ from $L_I$ to the generators of 
  $I$. For each ideal $J_B$ added to $\M$, which is generated by 
  $G(I) \cup \{x^B\}$: update $f_{J_B}(t)$ to $f_I(t)-(1-t)^{l_B}
  t^{d_B}$, compute $sd(f_{J_B})$, and set $L_{J_B}$ to 
  $\{x^Ax_{l_A}^{sd(f_{J_B}) - d_A} : x^A \in L_I, x^B >_{lex} x^A\}$. 
  Do the following:
    \begin{itemize}
    \item If $l_B<n-1$ and $\Left (\frac{x_{l_B+1}}{x_{l_B}} x^B) 
    \subset I$, include $x^B x_{l_B}^{sd(f_J) - d_B - 1} x_{l_B+1}$ 
    in $L_{J_B}$.
    \item If $x_{l_B-1}|x^B$ and $\Left (\frac{x_{l_B}}{x_{l_B-1}} x^B) 
    \subset I$, include $x^B x_{l_B-1}^{-1} x_{l_B}^{sd(f_J) - d_B + 1}$ 
    in $L_{J_B}$.
    \end{itemize}
\end{itemize}
\item Return the set of ideals $\Set$.
\end{enumerate}
\end{alg}

\begin{proof} (\emph{Correctness})
Certainly, any ideal produced by the above process will be 
strongly stable (because we check that the ideal generated by 
$G(I) \cup \{x^B\}$ is strongly stable before adding the monomial 
$x^B$ to $L_I$) and saturated (because no monomials added to the 
set of generators will be divisible by the variable $x_n$), and it 
will have the desired Hilbert series (because the ideal is added 
to $\Set$ when the Hilbert series is correct).

We need to show that every saturated strongly stable is produced: 
specifically, for each ideal $I$ produced in the algorithm, $L_I$ 
contains every monomial $x^B$ which can be added (in the 
lexicographic order) to the ideal $I$ to produce a saturated 
strongly stable ideal, say $J$, generated by $G(I) \cup \{x^B\}$. 
Suppose that the ideal $J$ is strongly stable; then
$$ \left\{ \frac {x_i} {x_j} x^B : x_j|x^B, i<j \right\} \subset I 
\textrm{, so, in particular, } x^A = \frac {x_{l_B-1}} {x_{l_B}} 
x^B \in I. $$
Note that the monomial $x^A$ is the smallest monomial in the 
lexicographic order (in degree $d_B$), which must be contained in 
the ideal $I$ if $J$ is strongly stable. Turning this around, at 
most two new monomials, say $x^E$ and $x^F$, can be added to the 
generators of $I$ after the monomial $x^A$:
$$ x^E = \frac {x_{l_A+1}} {x_{l_A}} x^A \textrm{ (if $l_A < n-1$) 
and } x^F = \frac {x_{l_A}} {x_{l_A-1}} x^A \textrm{ (if $x_{l_A-1} | 
x^A$)}. $$
These monomials, $x^E$ and $x^F$, are precisely those which are 
included in $L_I$. The monomials $x^E$ and $x^F$ are added to $L_I$, 
provided that $\Left (x^E) \subset I$ or $\Left (x^F) \subset I$ so 
that the ideals generated by $G(I) \cup \{x^E\}$ and $G(I) \cup 
\{x^F\}$ are saturated and strongly stable. Thus, every monomial 
$x^B$, which can be added to the generators of an ideal $I$ to 
produce a saturated strongly stable ideal, appears in $L_I$, so 
the algorithm will generate all of the desired ideals.

The algorithm terminates for any given Hilbert series because each 
list $L_I$ is finite and, by \cite{GG}, there is an upper bound for 
the largest degree of a minimal generator of a saturated ideal that 
depends only on its Hilbert polynomial, which in turn is determined 
by the Hilbert series.
\end{proof}

We include an example to illustrate this algorithm.

\begin{ex}
Suppose we wish to find all saturated strongly stable ideals in 
$R=K[x_0,x_1,x_2,x_3,x_4]$ with Hilbert series 
$H_{R/I}(t) = \frac{1 - 6t^2 + 8t^3 - 3t^4}{(1-t)^5}$. Thus, the 
numerator of the Hilbert series is $1 - 6t^2 + 8t^3 - 3t^4$.

\begin{itemize}
\item We begin with the zero ideal. We compute 
$f_{(0)}(t) = 6t^2 - 8t^3 + 3t^4$ and $sd(f_{(0)}) = 2$ 
(because $6t^2$ is the smallest nonzero term in $f_{(0)}$). 
We add $I_1 = (x_0^2)$ to $\M$, update $f_{I_1}(t)$ to 
$f_{(0)}(t) - t^2 = 5t^2 - 8t^3 + 3t^4$, record $sd(f_{I_1}) = 2$, 
and set $L_{I_1}$ to $\{x_0x_1\}$.

\item We replace $I_1$ in $\M$ with a new ideal $I_2 = (x_0^2, x_0x_1)$. 
We update $f_{I_2}$ to $f_{I_1}(t) - (1 - t) t^2 = 4t^2 - 7t^3 + 3t^4$, 
record $sd(f_{I_2}) = 2$, and set $L_{I_2}$ to $\{x_0x_2, x_1^2\}$.

\item We replace $I_2$ in $\M$ with the two ideals $I_3 = 
(x_0^2, x_0x_1, x_0x_2)$ and $I_4 = (x_0^2, x_0x_1, x_1^2)$.
  \begin{itemize}
  \item $f_{I_3} = f_{I_2}(t) - (1 - t)^2 t^2 = 3t^2 - 5t^3 + 2t^4$, 
  $sd(f_{I_3}) = 2$, and $L_{I_3} = \{x_0x_3, x_1^2\}$
  \item $f_{I_4} = f_{I_2}(t) - (1 - t) t^2 = 3t^2 - 6t^3 + 3t^4$, 
  $sd(f_{I_4}) = 2$, and $L_{I_4} = \varnothing$ 
  (because $x_0x_2 >_{lex} x_1^2$ and $x_0x_2 \not\in I_4$ so 
  $x_1x_2$ cannot be added to $I_4$)
  \end{itemize}

\item We replace $I_3$ in $\M$ with the two ideals $I_5 = 
(x_0^2, x_0x_1, x_0x_2, x_0x_3)$ and $I_6 = 
(x_0^2, x_0x_1, x_0x_2, x_1^2)$. 
We ignore $I_4$ (because $L_{I_4} = \varnothing$).
  \begin{itemize}
  \item $f_{I_5} = f_{I_3}(t)-(1-t)^3 t^2 = 2t^2-2t^3-t^4+t^5$, 
  $sd(f_{I_5}) = 2$, and $L_{I_5} = \{x_1^2\}$
  \item $f_{I_6} = f_{I_3}(t) - (1 - t) t^2 = 2t^2 - 4t^3 + 2t^4$, 
  $sd(f_{I_6}) = 2$, and $L_{I_6} = \{x_1x_2\}$
  \end{itemize}

\item We replace $I_5$ in $\M$ with the ideal $I_7 = 
(x_0^2, x_0x_1, x_0x_2, x_0x_3, x_1^2)$, and we replace $I_6$ 
with the ideal $I_8 = (x_0^2, x_0x_1, x_0x_2, x_1^2, x_1x_2)$.
  \begin{itemize}
  \item $f_{I_7} = f_{I_5}(t) - (1 - t) t^2 = t^2 - t^3 - t^4 + t^5$, 
  $sd(f_{I_7}) = 2$, and $L_{I_7} = \{x_1x_2\}$
  \item $f_{I_8} = f_{I_6}(t) - (1 - t)^2 t^2 = t^2 - 2t^3 + t^4$, 
  $sd(f_{I_8}) = 2$, and $L_{I_8} = \{x_2^2\}$
  \end{itemize}

\item We replace $I_7$ in $\M$ with the ideal $I_9 = 
(x_0^2, x_0x_1, x_0x_2, x_0x_3, x_1^2, x_1x_2)$, and we replace $I_8$ 
with the ideal $I_{10} = (x_0^2, x_0x_1, x_0x_2, x_1^2, x_1x_2, x_2^2)$.
  \begin{itemize}
  \item $f_{I_9} = f_{I_7}(t) - (1 - t)^2 t^2 = t^3 - 2t^4 + t^5$, 
  $sd(f_{I_9}) = 3$, and $L_{I_9} = \{x_1x_3^2, x_2^3\}$ 
  (because we need to add monomials of degree 3)
  \item $f_{I_{10}} = f_{I_8}(t) - (1 - t)^2 t^2 = 0$ 
  (We do not need $sd(f_{I_{10}})$ or $L_{I_{10}}$.)
  \end{itemize}

\item We add $I_{10}$ to $\Set$, and we replace $I_9$ in $\M$ with 
the two ideals $I_{11} = \\
(x_0^2, x_0x_1, x_0x_2, x_0x_3, x_1^2, x_1x_2, x_1x_3^2)$ and 
$I_{12} = (x_0^2, x_0x_1, x_0x_2, x_0x_3, x_1^2, x_1x_2, x_2^3)$.
  \begin{itemize}
  \item $f_{I_{11}} = f_{I_9}(t) - (1 - t)^3 t^3 = t^4 - 2t^5 + t^6$, 
  $sd(f_{I_{11}}) = 4$, and $L_{I_{11}} = \{x_2^4\}$
  \item $f_{I_{12}} = f_{I_9}(t) - (1 - t)^2 t^3 = 0$ 
  (We do not need $sd(f_{I_{12}})$ or $L_{I_{12}}$.)
  \end{itemize}

\item We add $I_{12}$ to $\Set$, and we replace $I_{11}$ in $\M$ 
with the ideal $I_{13} = \\
(x_0^2, x_0x_1, x_0x_2, x_0x_3, x_1^2, x_1x_2, x_1x_3^2, x_2^4)$.
    \begin{itemize}
    \item $f_{I_{13}} = f_{I_{11}}(t) - (1 - t)^2 t^4 = 0$
    \end{itemize}

\item We add $I_{13}$ to $\Set$.
\end{itemize}

Thus, there are three saturated strongly stable ideals with the given 
Hilbert series:
$$(x_0^2, x_0x_1, x_0x_2, x_1^2, x_1x_2, x_2^2)$$ 
$$(x_0^2, x_0x_1, x_0x_2, x_0x_3, x_1^2, x_1x_2, x_2^3)$$
$$(x_0^2, x_0x_1, x_0x_2, x_0x_3, x_1^2, x_1x_2, x_1x_3^2, x_2^4)$$
\end{ex}

\section{Applications and related questions}
\label{sec:final}

We conclude by discussing  some questions that, we believe, deserve 
further investigation along with some initial results.

It is well known that saturated strongly stable ideals figure 
prominently in the combinatorial structure of the Hilbert scheme.
This motivates the following problem.

\begin{question}\label{count-ideals}
What is the number of saturated strongly stable ideals in $R$ with 
a given Hilbert polynomial $p$?

Is there an explicit formula or a generating function for this 
number that depends only on $p$ and the number of variables in $R$?
\end{question}

In an appendix to her thesis \cite{AR}, Reeves presents an 
algorithm for generating saturated strongly stable ideals with a 
given Hilbert polynomial. Also, another algorithm was proposed 
independently in a recent paper \cite{CLMR} by Cioffi, Lella, 
Marinari, and Roggero. We thank the authors for kindly pointing 
this out to us after we submitted the first version of this paper. 
We briefly discuss the differences between these algorithms.

\begin{rem} \label{comparison-of-algorithms}
Algorithm \ref{alg-stable-Hilbpoly} differs from the  algorithm 
presented by Reeves in \cite{AR}: Her algorithm first computes all 
Hilbert series associated to the desired Hilbert polynomial by pairs 
of contractions and expansions and then generates all saturated 
strongly stable ideals for each Hilbert series. A single Hilbert 
series or ideal may be generated a number of times in each of these 
steps. On the other hand, our algorithm directly creates all ideals, 
each in a unique way, building them in larger and larger rings. We 
also give direct methods for producing all Hilbert series to a 
particular Hilbert polynomial in Algorithm \ref{alg-lex-Hilbpoly} 
and all saturated strongly stable ideals with a particular Hilbert 
series in Algorithm \ref{alg-Hilbert-series} that appear more efficient.

Furthermore, Reeves uses special matrices to encode the set of 
monomial generators of a strongly stable ideal. On these matrices, 
a certain kind of elementary row operations is performed to compute 
other saturated strongly stable ideals with the same Hilbert series. 
One problem to be solved then is that the correspondence between 
such matrices encoding strongly stable ideals and the set of 
strongly stable ideals itself (in a fixed polynomial ring) is not a 
bijection. The elementary row operations used may produce matrices, 
which do not encode any saturated strongly stable ideal. Hence, one 
needs a special procedure within the algorithm to check whether or 
not a given matrix represents a saturated strongly stable ideal. To 
avoid this trial and error technique, we did not use these matrices.

The algorithm suggested in \cite{CLMR} is more similar to Algorithm 
\ref{alg-stable-Hilbpoly} in that it is recursive in the number of 
variables (and the degree of the Hilbert polynomial). However,  instead 
of increasing the degrees of the minimal generators to achieve the 
correct Hilbert polynomial, a number of new generators are added to 
make the Hilbert function as large as possible in a fixed degree. 
Certain generators are then removed in all possible combinations 
to produce the desired saturated strongly stable ideals.

Observe that our  approach has the advantage of allowing us to 
estimate the number of steps to produce an ideal with a given 
Hilbert polynomial (see Theorem \ref{link-all-ideals}).
\end{rem}

In Table 1 we present some experimental results for the number of
strongly stable ideals with a given Hilbert polynomial in a given
polynomial ring. Recall that the Hilbert polynomial is actually 
the Hilbert polynomial of the quotient by the ideal.

Table 1 illustrates that, fixing the Hilbert polynomial, the number 
of strongly stable ideals in a polynomial ring with $n+1$ variables 
having this Hilbert polynomial increases with $n$ initially until 
it becomes stable and independent of $n$. This is indicated by the 
rightmost column in the table.

Our next result  explains this observation.

\begin{prop} \label{min-num-vars}
If $p(z)$ is a Hilbert polynomial, written as in  Equation 
\eqref{formula-Hilbert-polynomial}, then the number of saturated 
strongly stable ideals with Hilbert polynomial $p(z)$ in 
$R=K[x_0,\ldots,x_n]$ is the same whenever $n \geq b_0 + d - 1$.
\end{prop}

\newpage

\begin{longtable}
{|c|c|r|r|r|r|r|} \hline
$p(z)$ & ${a_0,a_1,a_2}$ & $n = 3$ & $n = 6$ & $n = 9$ & $n = 12$ & $n \gg 0$ \\
\hline
   $4$ & $ 4,0,0$ &    3 (0) &    3  (0) &     3  (0) &     3  (0) &     3 \\
   $8$ & $ 8,0,0$ &   12 (0) &   19  (0) &    20  (0) &    20  (0) &    20 \\
  $12$ & $12,0,0$ &   44 (0) &  104  (0) &   117  (0) &   119  (0) &   119 \\
  $16$ & $16,0,0$ &  143 (0) &  504  (0) &   617  (1) &   640  (1) &   644 \\
  $20$ & $20,0,0$ &  425 (0) & 2262  (2) &  3034  (4) &  3223  (6) &  3271 \\
  $24$ & $24,0,0$ & 1193 (1) & 9578 (16) & 14140 (46) & 15425 (61) & 15818 \\
\hline
 $4z+2$ & $ 4,4,0$ &    14 (0) &    28  (0) &    28  (0) &    28  (0) &    28 \\
 $4z+6$ & $ 8,4,0$ &    94 (0) &   394  (0) &   433  (0) &   434  (0) &   434 \\
$4z+10$ & $12,4,0$ &   469 (0) &  3702  (2) &  4536  (3) &  4627  (5) &  4632 \\
$4z+14$ & $16,4,0$ &  1939 (1) & 27486 (28) & 37792 (60) & 39462 (73) & 39677 \\
\hline
$8z-16$ & $ 4,8,0$ &   10 (0) &    18 (0) &    18  (0) &    18  (0) &    18 \\
$8z-12$ & $ 8,8,0$ &   66 (0) &   213 (0) &   232  (0) &   233  (0) &   233 \\
$8z- 8$ & $12,8,0$ &  347 (0) &  1911 (1) &  2268  (2) &  2310  (2) &  2313 \\
$8z- 4$ & $16,8,0$ & 1576 (0) & 14490 (7) & 18812 (18) & 19510 (32) & 19607 \\
\hline
$2z^2+ 6$ & $ 4,0,4$ &   3 (0) &    18 (0) &    19  (0) &    19  (0) &    19 \\
$2z^2+10$ & $ 8,0,4$ &  12 (0) &   224 (0) &   268  (0) &   271  (0) &   271 \\
$2z^2+14$ & $12,0,4$ &  44 (0) &  2073 (1) &  2835  (2) &  2930  (3) &  2938 \\
$2z^2+18$ & $16,0,4$ & 143 (0) & 15883 (9) & 24927 (32) & 26468 (63) & 26687 \\
\hline
$2z^2+4z-12$ & $ 4,4,4$ &  14 (0) &   45 (0) &    46 (0) &    46  (0) &    46 \\
$2z^2+4z- 8$ & $ 8,4,4$ &  94 (0) &  776 (0) &   868 (0) &   872  (1) &   872 \\
$2z^2+4z- 4$ & $12,4,4$ & 469 (0) & 9165 (4) & 11417 (9) & 11636 (15) & 11649 \\
\hline
$2z^2+8z-46$ & $ 4,8,4$ &  10 (0) &   37 (1) &   38 (1) &   38  (2) &   38 \\
$2z^2+8z-42$ & $ 8,8,4$ &  66 (0) &  588 (1) &  667 (1) &  671  (2) &  671 \\
$2z^2+8z-38$ & $12,8,4$ & 347 (0) & 6535 (3) & 8281 (8) & 8464 (13) & 8476 \\
\hline
$4z^2-16z+40$ & $ 4,0,8$ &   3 (0) &   18 (0) &   19 (0) &   19 (0) &   19 \\
$4z^2-16z+44$ & $ 8,0,8$ &  12 (0) &  224 (0) &  268 (0) &  271 (0) &  271 \\
$4z^2-16z+48$ & $12,0,8$ &  44 (0) & 2073 (1) & 2835 (3) & 2930 (5) & 2938 \\
\hline
$4z^2-12z+ 6$ & $ 4,4,8$ &  14 (0) &   45 (0) &    46  (0) &    46  (0) &    46 \\
$4z^2-12z+10$ & $ 8,4,8$ &  94 (0) &  761 (0) &   853  (1) &   857  (1) &   857 \\
$4z^2-12z+14$ & $12,4,8$ & 469 (0) & 8662 (4) & 10851 (13) & 11069 (16) & 11082 \\
\hline
$4z^2-8z-44$ & $ 4,8,8$ &  10 (0) &   37 (0) &   38 (1) &   38  (1) &   38 \\
$4z^2-8z-40$ & $ 8,8,8$ &  66 (0) &  588 (1) &  667 (1) &  671  (1) &  671 \\
$4z^2-8z-36$ & $12,8,8$ & 347 (0) & 6523 (3) & 8269 (7) & 8452 (12) & 8464 \\
\hline 
\caption[Number of strongly stable ideals]{The number (and time of
computation in seconds) of saturated strongly stable ideals with a 
given Hilbert polynomial, $p(z)$, in $K[x_0, \ldots, x_n]$ for 
several values of $n$}
\end{longtable}

\begin{proof}
The first expansion, the expansion of  1 in $R^{(d)}$,  gives 
$(x_0, \ldots, x_{n-d-1})$, an ideal with $n-d$ variables. By 
Theorem \ref{link-all-ideals}, the number of the remaining 
expansions will be at most $b_0-1$ (and depends upon how the 
expansions are chosen). It follows that the max index of any 
expanded monomial is at least $n-d-b_0 +1$. Hence, if 
$b_0-1 \leq n-d$, then the number of saturated strongly stable 
ideals generated is not constrained by the number of variables.
\end{proof}

The bound on the number of variables given in the last result 
is optimal in some cases.

\begin{ex} \label{ex:optimality-variable-bound}
(i) Fix integers $d \geq 0$ and $b_0 \geq 1$. Consider the 
saturated strongly stable ideals $I$ of $R = K[x_0,\ldots,x_n]$ 
with Hilbert polynomial $$ p(z) = \binom{z+d}{d} + b_0 - 1. $$
Then, using the notation of Theorem \ref{L_p}, $a_0 = b_0 - 1,\; 
a_1 = \cdots = a_{d-1} = 0$, and $a_d =1$. Following Algorithm 
\ref{alg-stable-Hilbpoly}, the first expansion will produce the 
ideal $I^{(d)} = (x_0, \ldots, x_{n-d-1}) \subset R^{(d)}$. The 
remaining $b_0-1$ expansions all occur in $R$. If $n=b_0+d-1$, 
then expanding all of the $n-d = b_0 - 1$ variables will produce 
a saturated strongly stable ideal with the desired Hilbert 
polynomial that is generated by quadrics. However, if $n \leq b_0+d-2$, 
then any $b_0 - 1$ expansions of $I^{(d)}$ will produce an ideal 
having a minimal generator whose degree is at least 3. Hence the 
bound on $n$ in Proposition \ref{min-num-vars} is optimal for this 
Hilbert polynomial.

(ii) Not every Hilbert polynomial will achieve this bound. Consider 
$p(z)=3z (=\binom{z+1}{2}-\binom{z-2}{2}+\binom{z}{1}-\binom{z-3}{1}$). 
If $n \geq 2$, there is exactly one saturated strongly stable ideal 
for this Hilbert polynomial even though $b_0 + d - 1 = 3$.
(The Hilbert polynomial of the ideal generated by 
$(x_0,\ldots,x_{n-3},x_{n-2}^3)$ is $p(z) = 3z$, while the Hilbert 
polynomial of the ideal $(x_0,\ldots,x_{n-4},x_{n-3}^2,x_{n-3} 
x_{n-2}, x_{n-2}^2)$ is $3 z + 1 \neq p(z)$)
\end{ex}

It is known that the lexicographic ideal has the worst
Castelnuovo-Mumford regularity among all saturated ideals with a
fixed Hilbert polynomial (see  \cite{GG}, \cite{DB}, and
\cite{AR2}). Theorem \ref{link-all-ideals} provides a quick new
argument. It also allows us to discuss the extremal ideals. We
denote by $\gin I$ the generic initial ideal of the ideal $I$ with
respect to the reverse lexicographic order.

\begin{thm}
  \label{thm:worst-reg}
Let $I \neq R$ be a saturated homogenous ideal of $R$. Write the 
Hilbert polynomial, $p$, of $R/I$ as in Equation 
\eqref{formula-Hilbert-polynomial}. Then the  Castelnuovo-Mumford 
regularity of $I$ satisfies $$ \reg I \leq b_0. $$
Furthermore, if $I$ is strongly stable, then equality is true if 
and only if $I = L_p$.

Moreover, if $I$ is any saturated homogenous ideal and 
$\charf K = 0$,  then $\reg I = b_0$ if and only if 
$\gin I = L_P$ and $I$ is of the form
\begin{equation}
  \label{eq:gens-extre-ideal}
I = (l_0,\ldots,l_{n-d-2},
f_d l_{n-d-1}, f_d f_{d-1} l_{n-d},\ldots, f_d \ldots f_{t+1}
l_{n-t-2}, f_d \ldots f_t)
\end{equation}
where $0 \leq t \leq d$, every $f_i \neq 0$ is a homogenous 
polynomial of degree $a_i \geq 0,\; a_n, a_t \geq 1$,  every 
$l_i$ is a linear form, and $I$ has (as indicated) $n+1-t$ minimal 
generators.  (Note that when $n = t$ the  ideal $I$ is simply 
defined as $I = (f_d)$.)
\end{thm}

\begin{proof}
First, we show the claims when $I$ is a strongly stable ideal. The
Eliahou-Kervaire resolution shows that the regularity of $I$ is the
maximal degree of a minimal generator of $I$. By Theorem
\ref{link-all-ideals} we know that $I$ can be obtained from the
ideal $(1) = R^{(d)}$ by at most $b_0$ expansions. Since each
expansion replaces a monomial by monomials whose degree is one more,
it follows immediately that the degrees of the minimal generators of
$I$ are at most $b_0$.

In order to characterize equality we use induction on $b_0 \geq 1$.
If  $b_0 = 1$, then $I$ is generated by linear forms, and the claim
follows. Let $b_0 > 1$, and assume that $I$ has a minimal generator
of degree $b_0$. Then, by the above argument, $I$ must have been
obtained from the ideal $(1) = R^{(d)}$ by exactly $b_0$ expansions.
Denote by $J'$ the ideal obtained by the first $b_0 - 1$ expansions,
and put $J = J' R$. Then $J$ must have a minimal generator of degree
$b_0 - 1$. Write the Hilbert polynomial of $R/J$ as
\begin{equation*}
p'(z)=\sum_{i=0}^{d} \left[\binom{z+i}{i+1}-\binom{z+i-b'_i}{i+1}\right].
\end{equation*}
Then $b_0' = b_0 - 1$. Hence, the induction hypothesis provides that 
$J = L_{p'}$. It follows that among the minimal generators of $J'$ 
having degree $b_0 -1$ only the smallest one in the lexicographic 
order is expandable. Expanding it, we get $I = L_p$ (see Remark 
\ref{rem-expan-lex-ideal}).

Second, let $I$ be an arbitrary saturated homogenous ideal with the
given Hilbert polynomial. Passing from $I$ to the almost lexsegment
ideal $I^*$ with the same Hilbert function as $I$ can only increase
the regularity by a result of Bigatti, Hulett, and Pardue (see
\cite{AB}, \cite{HH}, \cite{KP}). Since almost lexsegment ideals are
strongly stable we get $\reg I \leq \reg I^* \leq b_0$.

Finally, assume that the base field $K$ has characteristic zero.
Then   $\gin I$ is strongly stable and has the same regularity as
$I$ by \cite{BS}.  Hence, by the first part of the proof, $\reg I =
b_0$ if and only if $\gin I = L_p$. The claimed description of $I$
in this case now follows by Theorem 4.4 and Lemma 3.4 in
\cite{reg-extdeg}.
\end{proof}

Combined with the main result of Murai and Hibi in
\cite{MH-Gotzmann}, we obtain the following consequence. We would
like to thank Jeff Mermin for pointing this out.

Recall that a homogeneous ideal $I$ of $R = K[x_0,\ldots,x_n]$ is a
Gotzmann ideal if it has  as many minimal generators as the
lexsegment ideal $L_h \subset R$ corresponding to the Hilbert
function of $I$. Notice that an ideal $I$ of $R$ is saturated if it
has at most $n$ minimal generators.

\begin{cor}
  \label{cor:Gotzmann}
Let $I \subset R$ be a saturated homogeneous ideal, where $\charf K
= 0$. Write the Hilbert polynomial of $R/I$ as in Equation
\eqref{formula-Hilbert-polynomial}. Then the following conditions
are equivalent:
\begin{itemize}
  \item[(a)] $\reg I = b_0$;

  \item[(b)] $\gin I$ is a lexicographic ideal;

  \item[(c)] $I$ is a Gotzmann ideal with at most $n$ minimal
  generators;

  \item[(d)] $I$ is an ideal of the form as specified in Equation
  \eqref{eq:gens-extre-ideal}.
\end{itemize}
\end{cor}

\begin{proof}
Conditions (a), (b), and (d) are equivalent by Theorem
\ref{thm:worst-reg}. The equivalence to Condition (c) follows by
  Theorem 1.1 in \cite{MH-Gotzmann} because (d) shows that $I$ is a
  canonical
  critical ideal up to a coordinate transformation.
\end{proof}

We conclude with a crude estimate on the number of strongly stable
ideals with a given Hilbert polynomial.

\begin{cor}
  \label{cor:number-strongly-stable}
Let $p$ be the Hilbert polynomial of a graded quotient of $R$. 
Using the notation of Theorem \ref{L_p}, put $c = 
\min\{n, b_0 +d -1\}$. Then the number of saturated strongly stable 
ideals in $R$ with Hilbert polynomial $p$ is at most
$$ \binom{\binom{c-d + b_0 -1}{b_0 - 1} + 1}{a_d} 
\binom{\binom{c- (d -1) + b_0 -1}{b_0 - 1} + 1}{a_{d-1}} \cdot 
\ldots \cdot \binom{\binom{c + b_0 -1}{b_0 - 1} + 1}{a_0}. $$
\end{cor}

\begin{proof}
Assume first that $n \leq b_0 + d -1$, that is, $c = n$.

Using the notation of Theorem \ref{link-all-ideals}, it takes at 
most $a_j$ expansions to take $I^{(j+1)} \cdot R^{(j)}$ to $I^{(j)}$. 
By Theorem \ref{thm:worst-reg}, the degree of each expanded monomial 
is at most $b_0 -1$. Moreover, we expand only monomials in 
$K[x_0,\ldots,x_{n-j-1}]$. There are $N_j=\binom{n-j+b_0-1}{b_0-1}$ 
such monomials whose degree is at most $b_0 - 1$. For expanding at 
most $a_j$ of them, there are at most
$$ \binom{N_j}{0} + \binom{N_j}{1} + \cdots + \binom{N_j}{a_j} = 
\binom{N_j + 1}{a_j} $$
possibilities. Since we take $I^{(j+1)}$ to $I^{(j)}$ for $j = 
d,d-1,\ldots,0$, the claim follows in this case.

Second, if $n \geq b_0 + d-1$, then the number of strongly stable 
ideals is the same as for $n = b_0 + d-1$ by Proposition 
\ref{min-num-vars}. This concludes the argument.
\end{proof}



\begin{thebibliography}{ll}

\bibitem{DB}
{\sc D.~Bayer}, {\em The division algorithm and the Hilbert scheme},
Ph.D.\ Thesis, Harvard University, 1982.

\bibitem{BS}
{\sc D.~Bayer, M.~Stillman}, {\em A criterion for detecting
m-regularity}, Invent.\ Math.\ {\bf 87} (1987),  1--11.

\bibitem{AB}
{\sc A.~Bigatti}, {\em Upper bounds for the Betti numbers of a given
Hilbert function}, Comm.\ Algebra {\bf 21} (1993), 2317--2334.

\bibitem{CM}
{\sc G.~Caviglia, S.~Murai}, {\em Sharp upper bounds of the Betti
numbers for a given Hilbert polynomial}, Preprint, 2010
(arXiv:1010.0457v1).

\bibitem{CLMR}
{\sc F.~Cioffi, P.~Lella, M.~G.~Marinari, M.~Roggero}, {\em 
Segments and Hilbert schemes of points}, Discrete Math. {\bf 311}
(2011), 2238--2252.

\bibitem{DE}
{\sc D.~Eisenbud}, \emph{Commutative Algebra with a View Toward
Algebraic Geometry}. Graduate Texts in Mathematics {\bf 150},
Springer, 1995.

\bibitem{EK}
{\sc S.~Eliahou, M.~Kervaire}, {\em Mininimal resolutions of some
monomial ideals}, J.~Algebra {\bf 129} (1990), 1--25.

\bibitem{FMS}
{\sc C.~Francisco, J.~Mermin, J.~Schweig}, {\em Borel generators},
J.\ Algebra {\bf 332} (2011), 522--542.

\bibitem{KG}
{\sc K.~Gehrs}, {\em On stable monomial ideals}, Diploma Thesis,
University of Paderborn, 2003.

\bibitem{GG}
{\sc G.~Gotzmann}, {\em Eine Bedingung f\"ur die Flachheit und das
Hilbertpolynom eines graduierten Ringes},  Math.\ Z.\ {\bf 158}
(1978),  61--70.

\bibitem{GS}
{\sc D.~Grayson, M.~Stillman}, {\em Macaulay2, a software system 
for research in algebraic geometry}, Available at
{http://www.math.uiuc.edu/Macaulay2/}.

\bibitem{MLG}
{\sc M.~L.~Green}, {\em Generic initial ideals}, In J.~Elias,
J.M.~Giral, R.M.~Mir\'o-Roig, S.~Zarzuela (eds.), \emph{Six Lectures
on Commutative Algebra}, Progress in Mathematics {\bf 166},
Birkh\"auser, 1998,  119--186.

\bibitem{MLG2}
{\sc M.~L.~Green}, {\em Restrictions of linear series to hyperplanes
and some results of Macaulay and Gotzmann}, In E.~Ballico and
C.~Ciliberto (eds.), {\em Algebraic curves and projective geometry},
LNM {\bf 1389}, Springer, 76--86.

\bibitem{RH}
{\sc R.~Hartshorne}, {\em Connectedness of the Hilbert Scheme},
Inst.~Hautes~\'Etudes~Sci.~Publ.~Math.\ {\bf 29} (1966), 261--309.

\bibitem{RH2}
{\sc R.~Hartshorne}, {\em Algebraic Geometry}, Graduate Texts in
Mathematics {\bf 52}, Springer, 1977.

\bibitem{HeHi}
{\sc J.~Herzog, T.~Hibi}, \emph{Monomial ideals}. Graduate Texts in
Mathematics {\bf 260}, Springer, 2011.

\bibitem{HH}
{\sc H.~Hulett}, {\em Maximum Betti numbers of homogeneous ideals
with a given Hilbert function}, Comm.\ Algebra {\bf 21} (1993),
2335--2350.

\bibitem{MH}
{\sc S.~Murai, T.~Hibi}, {\em The depth of an ideal with a given Hilbert
function}, Proc. Amer. Math. Soc. 136 (2008), no. 5, 1533--1538.

\bibitem{MH-Gotzmann}
{\sc S.~Murai, T.~Hibi}, {\em  Gotzmann ideals of the polynomial
ring}, Math.\ Z.\ {\bf 260} (2008), 629--646.

\bibitem{reg-extdeg}
{\sc U.~Nagel}, {\em Comparing Castel\-nuovo-\-Mumford regularity and
extended degree: the  borderline cases}, Trans.\ Amer.\ Math.\ Soc.\
{\bf 357} (2005), 3585--3603.

\bibitem{NS}
{\sc R.~Notari, M.~L.~Spreafico}, {\em A stratification of Hilbert
schemes by initial ideals and applications},  Manuscripta Math.\
{\bf 101} (2000),  429--448.

\bibitem{KP}
{\sc K.~Pardue}, {\em Deformation classes of graded modules and
maximal Betti numbers}, Illinois J.\ Math.\ {\bf 40} (1996),
564--585.

\bibitem{AR}
{\sc A.~Reeves}, {\em On the combinatorial structure of the Hilbert
Scheme}, Ph.D.~thesis, Cornell University, 1992.

\bibitem{AR2}
{\sc A.~Reeves}, {\em The radius of the Hilbert scheme},
J.~Algebraic  Geom.\ {\bf 5} (1995), 639--657.

\bibitem{RS}
{\sc A.~Reeves, M.~Stillman}, {\em Smoothness of  the lexicographic
point}, J.~Algebraic Geom.\ {\bf 6} (1997), 235--246.

\bibitem{GV}
{\sc G.~Valla}: {\em On the Betti numbers of perfect ideals},
Compositio Math.\ {\bf 91} (1994), 305--319.

\end{thebibliography}
\end{document}